\documentclass[11pt,reqno]{amsart}
\usepackage{hyperref}\hypersetup{colorlinks=true, citecolor=blue}
\usepackage[dvipsnames]{xcolor}
\usepackage{
	amsfonts,
	amsmath,
	amssymb,
	epsfig,
	mathtools,
	esint,
	ifthen,
	color,
	xcolor,
	soul,
	empheq,
	mathrsfs,
	bm,
	graphicx,
	xfrac,
	enumitem,
	upgreek,
	a4wide
}

\newtheorem{theorem}{Theorem}[section]
\newtheorem{lemma}[subsection]{Lemma}
\newtheorem{proposition}[subsection]{Proposition}
\newtheorem{corollary}[subsection]{Corollary}
\theoremstyle{definition}

\theoremstyle{remark}

\numberwithin{equation}{section}

\newcommand{\R}{\mathbb{R}}
\newcommand{\eps}{{\varepsilon}}
\newcommand{\Lip}{{\mathrm{Lip}}}

\newcommand{\dist}{{\textup {dist}}}

\newcommand{\de}{\partial}

\renewcommand{\and}{\quad \text{and} \quad}

\renewcommand{\div}{\textup{div}}

\title[Nonlinear thin and boundary obstacle problems]{Regularity of solutions to
nonlinear thin and boundary obstacle problems}

\author[L.~Di Fazio]{Luca Di Fazio}
\address{La Sapienza Universit\`a di Roma}
\curraddr{Piazzale Aldo Moro 5, 00185 Roma (Italy)}
\email{difazio@mat.uniroma1.it}
\email{spadaro@mat.uniroma1.it}
\author[E.~Spadaro]{Emanuele Spadaro}
\thanks{
The authors have been supported by the ERC-STG Grant n. 759229
HiCoS ``Higher Co-dimension Singularities: Minimal Surfaces and 
the Thin Obstacle Problem''.}

\subjclass[2010]{35R35, 49Q05}

\keywords{Thin obstacle problems, Boundary obstacle problems, free boundary, Signorini's problem, elliptic regularity}

\date{\today}

\newcommand{\tr}{\text{\rm tr}}
\newcommand{\osc}{\mathrm{osc}}
\newcommand{\supp}{\operatornamewithlimits{supp}}
\renewcommand{\div}{\mathrm{div}}

\begin{document}

\begin{abstract}
Variational inequalities with thin obstacles and Signorini-type boundary conditions are classical problems in the calculus of variations, arising in numerous applications. In the linear case many refined results are known, while in the nonlinear setting our understanding is still at a preliminary stage.

In this paper we prove $C^1$ regularity for the solutions to a general class of quasi-linear variational inequalities with thin obstacles and $C^{1, \alpha}$ regularity for variational inequalities under Signorini-type conditions on the boundary of a domain.
\end{abstract}

\maketitle

\section{Introduction}\label{Introduction}
In this paper we prove the one-sided continuity of the gradient of the 
solutions to quasi-linear variational inequalities with thin obstacles
\begin{equation}\label{e.var-ineq}
	\int_{\Omega} \langle F(x,u,\nabla u),\nabla v - \nabla u \rangle
	+ F_0(x,u,\nabla u) (v -u) \ge 0 \quad \forall \, v \in \mathcal{K},
\end{equation}
where the solution $u$ is itself a member of $\mathcal{K}$, that is one of the following two sets:
\begin{itemize}
	\item
		{\bf Interior thin obstacles}
		\begin{equation}\label{K - ostacolo sottile}
			\mathcal{K} := \left\{v\in W^{1,\infty}(\Omega) \;:\; v\vert_{\de\Omega}= g,  v\vert_{\Sigma} \ge \psi\right\},
		\end{equation}
		where $\Sigma\subset\Omega$ is a smooth hypersurface dividing $\Omega$ into two connected components, $\Omega \setminus \Sigma=\Omega^+ \cup \Omega^-$, $g\in W^{1,\infty}(\de \Omega)$ a given boundary value and $\psi:\Sigma\to \R$;
	\item
		{\bf Boundary obstacles}
		\begin{equation}\label{K - ostacolo al bordo}
			\mathcal{K} := \left\{v\in W^{1,\infty}(\Omega) \;:\; v\vert_{\de\Omega} \ge \psi\right\},
		\end{equation}
		with the unilateral constraint given on the boundary of $\Omega$ by a function $\psi:\de\Omega \to \R$.
\end{itemize}
Here $F = (F_1, \ldots, F_{n+1}): \Omega\times \R\times \R^{n+1}\to \R^{n+1}$ and $F_0:\Omega\times \R\times \R^{n+1}\to \R$, 
$\Omega\subset \R^{n+1}$ a bounded open set with smooth boundary.

\medskip

The boundary variational inequalities are also known as Signorini's problem in the theory of elasticity (see, e.g., \cite{Fichera64} for more details on the physical background).
A natural case of a nonlinear variational inequality is that of minimal surfaces forced to lie above an obstacle which is prescribed on the boundary, as introduced by Nitsche \cite{Nitsche69} in a particular instance and previously by H. Lewy \cite{Lewy68}, who was able to analyzed the linearized problem with the Laplace operator.
More in general, variational inequalities of this kind might arise from minimization problems
\begin{equation}\label{e.minimization}
	\textup{minimize }\quad \int_{\Omega} h(x, u , \nabla u)\; dx
	\qquad u \in \mathcal{K},
\end{equation}
which lead to the variational inequality \eqref{e.var-ineq} with $F = \nabla_p h$ and $F_0 = \de_z h$, where we denote by 
$(x, z, p)\in \Omega\times \R\times \R^{n+1}$ the variables of $h$.

\medskip

This problem has been widely considered in the literature by numerous
authors: here we recall few of the earlier contributions which are more relevant for the present paper by Fichera \cite{Fichera64}, Lewy \cite{Lewy68, Lewy70}, Nitsche \cite{Nitsche69}, 
Giusti \cite{Giusti71, Giusti73,GiustiCIME}, Frehse \cite{Frehse75,Frehse77},
Kinderlehrer \cite{Kinderlehrer71, Kinderlehrer81}, 
Richardson \cite{Richardson78}, Caffarelli \cite{Caffarelli79},
Ural'tseva \cite{Uraltseva85,Uraltseva89}, only to mention a few (an increasing number of articles on variational inequalities with thin obstacles appeared in the recent years).
Under general conditions on the functions $F$, $F_0$ and on the domain $\Omega$, the existence of Lipschitz solutions has been established
(see, e.g., the works by Nitsche and Giusti \cite{Nitsche69,Giusti73, GiustiCIME} for the case of minimal surfaces and Giaquinta-Modica \cite{GiaMo75} for more general nonlinearities).

As far as further regularity of the solutions is investigated, in accordance with the linear case the one-sided continuity of the derivatives of the solutions up to the thin obstacle is expected.
Nevertheless, this problem has remained open in this generality since the early works, though several significant results have appeared in the last years.
The main breakthroughs have been obtained for the linear case of the Laplace operator. It was well known that in in this instance the solutions could not be more regular than having $\frac12$-H\"{o}lder derivatives on both sides of the thin obstacle, and the optimal one-sided $C^{1,1/2}$ regularity was first established in dimension $n=1$ by Richardson \cite{Richardson78} and more recently by Athanasopoulos-Caffarelli \cite{AtCa04} extended to general dimensions (recall also the $C^{1,\alpha}$ regularity previously obtained by Caffarelli \cite{Caffarelli79}).
Starting from these pioneering works, the H\"older one-sided continuity of the derivatives of solutions has been also proven to hold for some classes of quasi-linear operators, in the two-dimensional case by Kinderlehrer \cite{Kinderlehrer81} and in general dimension by Ural'tseva \cite{Uraltseva85,Uraltseva89}.

However, for the general operators in \eqref{e.var-ineq} the one-sided continuity of the gradients is an open problem and the best available results in this regards have been obtained by Frehse in a pair of pioneering papers \cite{Frehse75,Frehse77} which establish the continuity (with a logarithmic modulus of continuity) of the gradient of the solutions in dimension $n=1$, and the continuity of the tangential derivatives to the thin obstacles in general dimension $n\ge 2$.
As far as we known, those by Frehse are still the most general results, while more refined theorems are known for some specific operators, such as the minimal surface operator (see, e.g., \cite{Athanasopoulos83,FeSe17, FocardiSpadaro20}). 


\medskip

In this article we establish the $C^1$ and $C^{1,\alpha}$ regularity results for a general class of nonlinear variational inequalities \eqref{e.var-ineq}.
The main assumption we consider (apart from the regularity of the fields $F, F_0$) is the natural ellipticity condition:
\begin{itemize}
\item[(H)] the matrix $\left(\de_k F_i\right)_{ik}$ is uniformly positive definite in compact subsets.
\end{itemize}
This hypothesis is necessary to the existence of solutions,
e.g., for variational inequalities arising from minimization problems this
is nothing else than the convexity of the integrands in the last variable.

Building upon the pioneering works by Frehse \cite{Frehse75,Frehse77} and on Ural'tseva's approach based on De Giorgi's method \cite{Uraltseva85,Uraltseva89},
in this paper we show the following result.

\begin{theorem}\label{t.1}
Let $\Omega\subset\R^{n+1}$ be a bounded open set with $C^2$ boundary, $F:\Omega\times \R\times \R^{n+1}\to \R^{n+1}$ and $F_0:\Omega\times \R\times \R^{n+1}\to \R$ functions of class $C^1$ and the obstacle function $\psi$ in $\mathcal{K}$ is of class $C^2$.
Assume that the ellipticity condition (H) holds.
Then, 
\begin{itemize}
\item[(i)] every Lipschitz solution $u:\Omega\to \R$ to the thin obstacle problem \eqref{e.var-ineq}  with $\mathcal{K}$ given by \eqref{K - ostacolo sottile} 
has one-sided continuous derivatives up to the thin obstacle $\Sigma$: i.e.,
$u\in C^1(\Omega^+\cup \Sigma)\cap C^1(\Omega^-\cup \Sigma)$.

\item[(ii)] every Lipschitz solution $u:\Omega\to \R$ to the boundary variational inequality \eqref{e.var-ineq} with $\mathcal{K}$ given by \eqref{K - ostacolo al bordo} is $C^{1,\alpha}(\overline{\Omega})$ for some $\alpha\in (0,1)$.
\end{itemize}
\end{theorem}


This is the first result on the continuity of the derivatives of the solutions to the thin obstacle problems for fairly general nonlinear variational inequalities.
The case of linear operators $F_i(x, z, p) = \sum_{j=1}^{n+1}a_{ij}(x) p_j$
has been considered in \cite{Caffarelli79, Kinderlehrer81, Uraltseva85}
with weaker assumptions on the coefficients $a_{ij}$ from time to time (e.g., $a_{ij}\in W^{1,q}$ with $q>n+1$ are allowed in the work of Uralt'seva \cite{Uraltseva85} ).
In \cite{Uraltseva89} Uralt'seva considered also the case of quasilinear operators
$F_i(x, z, p) = \sum_{j=1}^{n+1}a_{ij}(x,z) p_j$ and proves $C^{1,\alpha}$ regularity of the solutions to Signorini's problem up to the boundary.

\medskip

In this paper we combine and extend the ideas developed for the minimal surface operator by Fern\'andez-Real and Serra \cite{FeSe17} in the context of parametric  solutions to thin obstacle problem according to De Giorgi's theory of Caccioppoli sets, and by Focardi and the second author \cite{FocardiSpadaro20} in the nonparametric setting.

The starting point is Frehse's general partial regularity result \cite{Frehse77}, which we use to perform a blowup analysis inspired by \cite{FocardiSpadaro20} in order to prove the $C^1$ regularity of the solutions to the general variational inequality.
We stress that in \cite{FocardiSpadaro20}, as well as in the works by Uralt'seva \cite{Uraltseva85, Uraltseva89} only the boundary obstacle problem is considered, where an additional constraint acts on the non-coincidence set of the solutions (i.e., the natural homogeneous Neumann condition on the co-normal derivative).
The extension of this analysis to the general case needs the introduction of new ideas, which in particular employs a comparison principle with paraboloids introduced in \cite{FeSe17}.
With these ingredients, we prove that blowups to the variational inequalities are flat, one-dimensional and unique, thus leading to the $C^1$ regularity around points of the free boundary.

Building upon it, we extend then the approach via De Giorgi's classes introduced by Uralt'seva \cite{Uraltseva85, Uraltseva89} in order to deduce the $C^{1, \alpha}$ regularity for the solutions to the boundary variational inequality.

\medskip

As far as we are aware, not much is known on the optimal regularity of the solutions and on the structure of the free boundary in the quasi-linear case, especially if compared to the linear case (see, e.g., \cite{AtCa04, CSV20, FRRO21, FocardiSpadaro16, FocardiSpadaro18, FocardiSpadaro18corrections, GarofaloPetrosyan09,KPS15, SY21}).
The only available results are those proven for minimal surfaces with thin analytic obstacles in dimension $n=1$ by Athanasopoulos \cite{Athanasopoulos83} and in general dimension for flat obstacle by Focardi and the second author \cite{FocardiSpadaro20}.

\section{Preliminaries}\label{s.rettificare}

\subsection{Reduction to flat boundaries and zero obstacles}
We use the following notation $x=(x',x_{n+1}) \in \R^n\times \R$ and 
for every $r>0$ we set
\begin{gather*}
	B_r = \{x \in \R^{n+1} \;:\; |x| < r\},
	\\ B_r^+ = B_r\cap \{x_{n+1}>0\},
	\quad
	B_r^- = B_r\cap \{x_{n+1}<0\},
	\quad
	B_r' = B_r\cap \{x_{n+1}=0\}.
\end{gather*}
In the following $\Sigma$ denotes the hypersurface where the thin obstacle is prescribed: i.e.,
\begin{itemize}
\item for the thin obstacle problem $\Sigma$ is a hypersurface splitting the domain 
$\Omega$ into two parts, $\Omega= \Omega^+ \cup \Omega^-$, with $\partial \Omega^+\cap \partial \Omega^-=\Sigma$;

\item for the boundary value problem  $\Sigma = \de \Omega$. In order to unify the following discussion, in this case we set 
$\Omega^+ = \Omega$.
\end{itemize}

Given a point $x_0 \in \Sigma$,
without loss of generality we can assume that locally around $x_0$ the hypersurface $\Sigma$ is given by the graph of a function
$\phi:\R^{n}\to \R$, i.e., there exists $R>0$ such that
\begin{gather*}
	\Omega^+ \cap B_R(x_0) = 
	\left\{ (x', x_{n+1})\in \R^n\times \R : x_{n+1}> \phi(x') \right\} \cap B_R(x_0).
\end{gather*}
In particular, the map $\Phi:\R^{n+1} \to \R^{n+1}$ defined by
\begin{gather*}
	\Phi(x',x_{n+1}) = x_0+\left(x', x_{n+1}+\phi(x')\right)
\end{gather*}
is a local diffeomorphism between a neighborhood of the origin, say $B_r$,
and a neighborhood of $x_0$, $U_0= \Phi(B_r)$, such that
\begin{gather*}
	\Phi(B_r^+)= \Omega^+ \cap U_0.
\end{gather*}
Since all the estimates we give are local, we always choose the coordinates
according to the diffeomorphism $\Phi$: given a solution $u$ to the variational inequality \eqref{e.var-ineq}, if we set $\bar u(x) = u(\Phi(x))$,
$\bar v(x) = v(\Phi(x))$, then
\begin{align}\label{e.var-ineq-3}
0 \leq	\int_{\Omega\cap U_0} &\langle F(y,u,\nabla u),\nabla v - \nabla u \rangle
	+ F_0(y,u,\nabla u) (v -u)\;dy\notag\\
	&=\int_{B_r^+}\langle \bar F(x,\bar u,\nabla \bar u),\nabla \bar v - \nabla \bar u \rangle + \bar F_0(x,\bar u,\nabla \bar u) (\bar v - \bar u)\;dx
\qquad \forall\; \bar v \in \mathcal{\bar K},
\end{align}
with 
\begin{gather*}
	\mathcal{\bar K} := \left\{w\in W^{1,\infty}(B_r^+) \;:\; w\vert_{B_r'} \ge \bar \psi,\;w\vert_{\partial B_r^+\setminus B_r'} = 	\bar u\vert_{\partial B_r^+\setminus B_r'} \right\},\\
\bar \psi(x') = \psi(\Phi(x',0)),\\
	\bar F(x,z,p) = 
	A(x)^{-1}F(\Phi(x),z,(A(x)^{-1})^Tp),\\
	\bar F_0 (x,z,p) = F_0 (\Phi(x),z,(A(x)^{-1})^Tp) ,\\
	A(x) = D\Phi(x).
\end{gather*}
Note that the ellipticity condition (H) for the associate operator
\begin{gather*}
	\bar H \bar u = - \div \left(\bar F(x,\bar u,\nabla \bar u)\right) + \bar F_0(x,\bar u,\nabla \bar u)
\end{gather*}
still holds true.

In a similar way, we can also subtract the obstacle from the solution $\bar u$:
setting $\tilde u(x)= \bar u (x) - \bar \psi (x')$, we get
\begin{align}
\int_{B_r^+}\langle \tilde F(x,\tilde u,\nabla \tilde u),\nabla \tilde v - \nabla \tilde u \rangle + \tilde F_0(x,\tilde u,\nabla \tilde u) (\tilde v - \tilde u)\;dx\geq 0,
\end{align}
for every $\tilde v \in \mathcal{\tilde K} = \left\{w\in W^{1,\infty}(B_r^+) : w\vert_{B_r'} \ge 0
\;w\vert_{\partial B_r^+\setminus B_r'} = 	\tilde u\vert_{\partial B_r^+\setminus B_r'}
\right\}$, with 
\begin{gather*}
\tilde F(x,z,p)= \bar F(x,z + \bar \psi(x'),p + \nabla \bar \psi(x')),\\
\tilde F_0(x,z,p)= \bar F_0(x,z + \bar \psi(x'),p + \nabla \bar \psi(x')),
\end{gather*}
still preserving the ellipticity condition (H).

\subsection{Hypotheses on $F$ and $F_0$}
In view of the discussion above, we can therefore assume what follows for the variational inequality \eqref{e.var-ineq}:
\begin{enumerate}[label=(H\arabic*), start=0]
	\item \label{Hyp. 0}
		$F \in C^{1}(U \times\R\times \R^{n+1} , \R^{n+1})$, $F_0\in C^{1}(U \times \R \times \R^{n+1} , \R)$,
where $U= B_1$ in the thin obstacle problem and $U=B_1^+$ for the boundary variational inequality;		
	\item \label{ellip}
		for every $M>0$ there exists $\lambda = \lambda(M) > 0$ such that
	\begin{gather*}
		\langle D_p F(x,z,p) \xi,\xi \rangle \ge  \lambda |\xi|^2
		\\ \forall \; (x,z,p) \in B_1 \times \R \times \R^{n+1},
		\quad |z|, \, |p| \le M,
		\quad \forall \; \xi \in \R^{n+1}.
	\end{gather*}
\end{enumerate}

The constants appearing in all the estimates of the subsequent sections might depend on the dimension $n$, the Lipschitz constant of the solutions $u$, the modulus of continuity of $F$ and $F_0$ and their first derivatives, and on the local ellipticity constant $\lambda$.

\subsection{Frehse's results}
We recall the results proven by Frehse in \cite{Frehse77} which are relevant for our analysis.

\begin{theorem}[\cite{Frehse77}] \label{Frehse}
	Under assumptions {\rm \ref{Hyp. 0}} and {\rm \ref{ellip}}, every Lipschitz solution $u$ of the variational inequality \eqref{e.var-ineq} for either the thin obstacle problem or the Signorini problem has continuous tangential derivatives: $\de_i u \in C(B_1^+ \cup B_1')$ for $i = 1,\dots,n$ with a local modulus of continuity $\omega$, 
	\begin{equation}\label{e.modulusFrehse}
		|\de_i u(x) - \de_i u(y)| \le \omega(|x-y|) \qquad 	\forall \; x,y \in B_r^+ \cup B_r',
	\end{equation}
	where $\omega(t) = C|\log t|^{-q}$ with $C=C(r, \Lip \, u) > 0$ and $q=q(n, r, \Lip \, u)>0$, $r \in (0,1)$.
Moreover, if $n=1$ the normal derivative is continuous too, thus implying that
$\nabla u= (\de_1 u ,\de_{2} u) \in C(B_1^+ \cup B_1',\R^2)$ with the same local  modulus of continuity \eqref{e.modulusFrehse} for a suitable choice of the constants $C,q$.
\end{theorem}

We also need the $H^2$ regularity of solutions proven by Frehse.

\begin{lemma}[{\cite[Lemma 2.2]{Frehse77}}]\label{stima H^2}
	Let $u$ be a solution to the variational inequality \eqref{e.var-ineq} for either the thin obstacle problem or the boundary variational inequality, under the assumptions {\rm \ref{Hyp. 0}} and {\rm \ref{ellip}}. There exists $C=C(\Lip \, u)>0$ such that for every $x_0\in B'_1$ and $0<2r<1-|x_0|$,  we have
	\begin{gather*}
		\int_{B^+_r(x_0)} |D^2 u|^2 \le {C \over r^2} \int_{B_{2r}^+(x_0)} |\nabla u|^2 + C r^{n+1}.
	\end{gather*}
\end{lemma}

\section{$C^1$ Regularity}

In this section we prove the $C^1$ regularity for the solutions $u\in \mathcal{A}_g$ to the variational inequalities with thin and boundary obstacles:
\begin{equation}\label{e.var-ineq-c1}
\int_{\Omega} \langle F(x,u,\nabla u),\nabla v - \nabla u \rangle
	+ F_0(x,u,\nabla u) (v -u) \ge 0 \quad \forall \, v \in \mathcal{A}_g,
\end{equation}
and
\begin{itemize}
\item	{\bf Interior thin obstacles}: $\Omega = B_1$ and
\begin{equation*}
\mathcal{A}_g := \left\{v\in W^{1,\infty}(B_1) \;:\; v\vert_{\de B_1}= g,  v\vert_{B_1'} \ge 0\right\},
\end{equation*}

\item	{\bf Boundary obstacles}: $\Omega = B_1^+$ and
\begin{equation}
\mathcal{A}_g := \left\{v\in W^{1,\infty}(B_1^+) \;:\; v\vert_{\de B_1^+\setminus B_1'}= g,  v\vert_{B_1'} \ge 0\right\},
\end{equation}
\end{itemize}
with $g\in W^{1,\infty}(\R^{n+1})$ is a given function.

The coincidence set and the free boundary of a solution $u$ are respectively the sets
\begin{gather*}
\Lambda(u) = \left\{(x',0)\in B_1' : u(x',0)=0 \right\},\\
\qquad \Gamma(u) = \left\{(x',0)\in \Lambda(u) : \forall \;r>0\;\exists\;(y',0)\in B_r'(x)\; u(y',0)>0 \right\},
\end{gather*}
i.e., $\Gamma(u)$ is the boundary of $\Lambda(u)$ in the relative topology of $B_1'$.

The main result is the following.

\begin{theorem}\label{t.c1}
Let $u$ be a Lipschitz solution to the variational inequality \eqref{e.var-ineq-c1} for either the interior thin or the boundary obstacle problem.
Then, $u \in C^1(B^+_1 \cup B'_1)$.
\end{theorem}

It is clear that Theorem \ref{t.1} (i) is a corollary of Theorem \ref{t.c1}
by following the local straightening of the obstacle explained in the 
previous section.

The proof of the $C^1$ regularity is made by a blowup analysis following the approach in \cite{FocardiSpadaro20}.
In particular, we proceed in three steps: first we show that the rescaled solutions of the variational inequality have a profile which is one-dimensional; then, by the maximum principle, we prove that around points of the free boundary the blowups are actually flat and unique; and finally, we show how the $C^1$ regularity follows from the existence of unique blowups.

The difference between the two obstacle problems is that for the boundary obstacle problem the natural homogeneous Neumann boundary conditions hold in the subset of $B_1'$ where the solution does not touch the unilateral constraint.
This fact imposes an additional constraint on the solutions which simplifies the analysis. This is what happens in \cite{FocardiSpadaro20}, but this is not the case for the thin obstacle problems, which needs new ideas.

\subsection{Classification of blowups: one-dimensional profiles}

Let $\{z_k\} \subset \Gamma(u)$, $\{t_k\} \subset \R$ such that $0<t_k < 1 - |z_k|$, $t_k \to 0$, $z_k \to z_0 \in \Gamma(u)$. We set
\begin{gather}\label{def di u_k}
	u_k(x) = {u(z_k + t_k x) \over t_k} \qquad \forall x \in B_1.
\end{gather}
We call $u_k$ a rescaling of $u$. Since we want to study the behavior of $u$ around $z_0$, we have to look at the limit of $u_k$. When $z_k=z_0$ for all $k$ and the
limit of the rescalings exists, we call it a {\bf blowup} of $u$ at $z_0$.
Note that $\Lip(u_k) = \Lip(u)$, therefore by Ascoli-Arzel\`a's theorem the set of rescalings is precompat in $L^\infty$.

The first lemma shows that the limits of the rescaled solutions depend only on the normal variable $x_{n+1}$.

\begin{lemma}\label{l:Gamma Conv}
Let $u_k$ be a sequence of rescalings as in \eqref{def di u_k} with
$z_k\to z_0$, $t_k\downarrow0$ and assume that $u_k \to u_\infty$ uniformly.
Then,
\begin{itemize}
\item for the {\bf thin obstacle problem} $u_\infty(x) = w(x_{n+1})$, with
\begin{gather*}
w(t) = 
\begin{cases}
a^+ t & t\ge 0,\\
a^- t & t\le 0
\end{cases}
\qquad \text{for some $a^+ \le a^-$;}
\end{gather*}
\item for the {\bf boundary obstacle problem} $u_\infty(x) = a x_{n+1}$ for some $a\in \R$ such that $F_{n+1}(z_0, 0, 0, a)\le 0$.
\end{itemize}
Moreover, the function $u_\infty$ is a solution
to the thin or the boundary obstacle problem	
\begin{gather*}
\int_{\Omega} \langle F(z_0,0,\nabla u_\infty),\nabla v_\infty -\nabla u_\infty \rangle \ge 0 \qquad \forall \; v_\infty \in \mathcal{A}_{u_\infty}.
\end{gather*}
\end{lemma}

\begin{proof}
We start by rescaling the variational inequality \eqref{e.var-ineq-c1}:
set $B_{k} := B_{t_k}(z_k)$ and let $w\in \mathcal{A}_{u_k}$,
then we choose
	\begin{gather*}
		v(y) = \begin{cases}
			u(y) & y \in \Omega \setminus B_{k},\\
			t_k w\left( {y-z_k\over t_k} \right)  & y \in B_{k},
		\end{cases}
	\end{gather*}
recalling that $\Omega=B_1$ or $\Omega=B_1^+$ for the thin and boundary obstacle problem, respectively.
It is straightforward to verify that $v \in \mathcal{A}_g$ so that, after a change of variables, we get
\begin{align*}
		\int_{\Omega} &\langle F(z_k+t_k x,t_k u_k,\nabla u_k),\nabla w -\nabla u_k \rangle \, +
		\\& + \int_{\Omega} t_k F_0(z_k+t_k x,t_k u_k,\nabla u_k) (w-u_k) \ge 0 \qquad \forall w \in \mathcal{A}_{u_k}.
	\end{align*}
Thus, $u_k$ is a solution to a rescaled problem and the associated rescaled operator is
	\begin{equation}\label{operatore_k}
		H_k u_k \equiv - \div \left(F(z_k+t_k x,t_k u_k,\nabla u_k)\right) + t_k F_0(z_k+t_k x,t_k u_k,\nabla u_k).
	\end{equation}
Now we fix $v_\infty \in \mathcal{A}_{u_\infty}$. For every $k\ge 1$, we define
	\begin{gather*}
		\varphi_k(x)=\begin{cases}
			1 & |x| \le 1-{1\over k},\\
			k^2-k(k+1)|x| & 1-{1\over k} \le |x| \le 1-{1\over k+1},\\
			0 & 1-{1\over k+1}\le |x|.
		\end{cases}
	\end{gather*}
	
We then choose $w = (1-\varphi_k)u_k + \varphi_k v_\infty \in \mathcal{K}_{u_k}$: from the variational inequality satisfied by $u_k$ we get that $\mbox{I}_k + \mbox{II}_k + \mbox{III}_k \ge 0$, with
	\begin{gather*}
		\mbox{I}_k = t_k \int_{\Omega} \varphi_k F_0(z_k+t_k x,t_k u_k,\nabla u_k) (v_\infty-u_k),\\
		\mbox{II}_k = \int_{\Omega} \langle F(z_k+t_k x,t_k u_k,\nabla u_k),\nabla \varphi_k \rangle (v_\infty - u_k),\\
		\mbox{III}_k = \int_{\Omega} \varphi_k \langle F(z_k+t_k x,t_k u_k,\nabla u_k),\nabla v_\infty -\nabla u_k \rangle.
	\end{gather*}
	
Now we want to compute the limits of the above quantities as $k \to +\infty$. First of all, since the integrand in $\mbox{I}_k$ is bounded uniformly on $k$ and $t_k\to0$, we deduce that $\mbox{I}_k \to 0$. Now we show that $\mbox{II}_k \to 0$ as well. For every $k\ge 1$, there exists $x_k \in B_1$ such that $1-{1 \over k} \le |x_k| \le 1-{1 \over k+1}$ and
	\begin{gather*}
		\sup\limits_{1-{1 \over k} \le |x| \le 1-{1 \over k+1}} |v_\infty(x)-u_k(x)| = |v_\infty(x_k)-u_k(x_k)|.
	\end{gather*}
Thus we have
	\begin{gather*}
		|\mbox{II}_k| \le C\,k(k+1)\int_{1-{1 \over k} \le |x| \le 1-{1 \over k+1}} |v_\infty(x) - u_k(x)| \le
		C_k |v_\infty(x_k)-u_k(x_k)|,
	\end{gather*}
with $C_k = C\,k(k+1)(|B_{1-{1 \over k+1}}| - |B_{1-{1 \over k}}|)$.
Note that  $C_k\to C (n + 1) \omega_{n + 1}$, therefore it is enough to show that $|v_\infty(x_k)-u_k(x_k)| \to 0$. For some subsequence (which we will not relabel) we have that $x_k \to x_\infty \in \de B_1$. Since $v_\infty$ and $u_\infty$ are continuous and agree at the boundary, and since $u_k$ converges uniformly to $u_\infty$, we have that
\begin{align*}
|v_\infty(x_k)-u_k(x_k)|&\leq |v_\infty(x_k)-v_\infty(x_\infty)| + |u_\infty(x_\infty)-u_k(x_\infty)| + |u_k(x_\infty)-u_k(x_k)|\\
&\leq \Lip(v_\infty) |x_k-x_\infty| + \|u_\infty-u_k\|_\infty + \Lip(u_k) |x_k-x_\infty|
\to 0,
\end{align*}
where we used that $\Lip(u_k) = \Lip(u)$.

Finally, we want to compute the limit of the quantity $\mbox{III}_k$. For this purpose, we set
\begin{gather*}
		\mbox{III}'_k = \int_{B_1} \langle F(z_k+t_k x,t_k u_k,\nabla u_k),\nabla v_\infty -\nabla u_k \rangle,
\end{gather*}
and we notice that $|\mbox{III}_k - \mbox{III}'_k| \to 0$ because
the integrand is uniformly bounded and $1-\varphi_k$ is supported in $B_1\setminus B_{1-1/k}$. To show that $\mbox{III}'_k$ converges, we need something better than uniform convergence. So we apply Lemma \ref{stima H^2} to $u$ so that, for every $k \ge 1$ such that $2t_k < 1 - |z_k|$, we have
	\begin{gather*}
		\int_{B^+_1} |D^2 u_k(x)|^2 \,dx = t_k^2 \int_{B^+_1} |D^2 u(z_k + t_k x)|^2 \,dx =\\
		= t_k^{1-n} \int_{B_{2t_k}^+(z_k)} |D^2 u(y)|^2 \,dy \le C\, t_k^{-1-n} \int_{B_{2t_k}^+(z_k)} |\nabla' u|^2 + C t_k^2 \le C.
	\end{gather*}
Therefore, $\{u_k\}$ is bounded in $H^2(B^+_1)$, and thus it has a weakly convergent subsequence in $H^2(B^+_1)$, which we will not relabel:
$u_k \rightharpoonup u_\infty$ in $H^2(B^+_1)$ and $u_k \to u_\infty$ in $H^1(B^+_1)$. Up to pass to further subsequences, we can also assume that $\nabla u_k \to \nabla u_\infty$ a.e. in $B_1^+$ and as a consequence
\begin{gather*}
F(z_k+t_k x,t_k u_k,\nabla u_k) \to F(z_0,0,\nabla u_\infty) \quad \mbox{in } L^2(B_1).
\end{gather*}
Indeed by {\rm \ref{Hyp. 0}}
\begin{gather*}
\int_{B_1} |F(z_k+t_k x,t_k u_k,\nabla u_k) - F(z_0,0,\nabla u_\infty)|^2 \le
		\\ \le C\int_{B_1} |z_k+t_k x - z_0|^2 + |t_k u_k|^2 + |\nabla u_k-\nabla u_\infty|^2 \to 0.
\end{gather*}
Finally, since also $\nabla u_k \to \nabla u_\infty$ in $L^2(B_1)$, we get
\begin{gather*}
\mbox{III}'_k \to \int_{B_1} \langle F(z_0,0,\nabla u_\infty),\nabla v_\infty -\nabla u_\infty \rangle.
\end{gather*}
	
We have indeed shown that $u_\infty$ is a solution to the thin or the boundary obstacle problem	
\begin{gather*}
\int_{\Omega} \langle F(z_0,0,\nabla u_\infty),\nabla v_\infty -\nabla u_\infty \rangle \ge 0 \qquad \forall \; v_\infty \in \mathcal{A}_{u_\infty},
\end{gather*}
with associated operator
\begin{gather*}
H_\infty u_\infty = - \div\left( F(z_0,0,\nabla u_\infty) \right).
\end{gather*}
By Frehse's Theorem \ref{Frehse} $\nabla'u(z_k) = 0$ and
\begin{gather*}
|\nabla' u_k(x)| = |\nabla' u(z_k + t_k x)-\nabla'u(z_k)|\le \omega(t_k) \to 0.
\end{gather*}
In other words, $\nabla' u_k \to 0$ uniformly on $B_1$, thus implying $\nabla' u_\infty \equiv 0$, i.e., $u_\infty(x) = w(x_{n+1})$ for some Lipschitz function $w$.

Considering the ellipticity condition {\rm \ref{ellip}} (applied with $\xi = e_{n+1}$)
\begin{gather*}
\de_{p_{n+1}} F_{n+1}(z_0,0,\nabla u_\infty) \ge \lambda>0,
\end{gather*}
we infer that $w$ is a linear function, i.e., there exists $a^+\in \R$ such that $w(t) = a^+ t$ for all $t\geq 0$.

As for the thin obstacle problem, we apply the same considerations to $B_1^-$,
infering the existence of $a^+, a^- \in \R$ such that
	\begin{equation*}
		w(x_{n+1}) = \begin{cases}
			a^+ x_{n+1} & x_{n+1}\ge 0,\\
			a^- x_{n+1} & x_{n+1}\le 0.
		\end{cases}
	\end{equation*}
Recalling that $u_\infty$ is a supersolution to $H_\infty$ in $B_1$ we have that, for every $\varphi \in C^\infty_0(B_1)$ with $\varphi \ge 0$,
	\begin{gather*}
		0\le \int_{B_1} \langle F(z_0,0,\nabla u_\infty),\nabla\varphi \rangle =
		\int_{B'_1}  (F_{n+1}(z_0,0,0,a^-) - F_{n+1}(z_0,0,0,a^+))\,\varphi,
	\end{gather*}
thus implying that 	
\begin{gather*}
	F_{n+1}(z_0,0,0,a^+)\le F_{n+1}(z_0,0,0,a^-),
\end{gather*}
and by ellipticity {\rm \ref{ellip}} ($\de_{p_{n+1}} F_{n+1}>0$) we
deduce $a^+ \leq a^-$.

Finally note that, for the boundary obstacle problem, for every $\varphi\geq 0$ we have that
\begin{gather*}
0\leq \int_{B_1^+} \langle F(z_0,0,\nabla u_\infty),\nabla \varphi \rangle = 
\int_{B_1^+} \langle F(z_0,0,0,a),\nabla \varphi \rangle = - \int_{B_1'} F_{n+1}(z_0,0,0,a) \varphi,
\end{gather*}
i.e., $F_{n+1}(z_0,0,0,a)\le 0$ which conclude the proof of the classification of the blowups for Signorini's problem.
\end{proof}

\subsection{Construction of barriers}
We say that a differential operator $H$ satisfying {\rm \ref{Hyp. 0}}, {\rm \ref{ellip}} is $t$-rescaled ($t >0$) if for every $M > 0$ there exists $L  = L(M) > 0$ such that
\begin{gather}
|-\div_x F(x,z,p) - \langle \de_z F(x,z,p), p \rangle + F_0(x,z,p)| \le t L,\notag\\
|D_p F(x,z,p)| \le L,\label{e.rescaled H}\\
\forall \; x \in B_1 \quad |z| \le M \quad |p| \le M\notag.
\end{gather}

We saw in the proof of Lemma \ref{l:Gamma Conv} that, if $u$ solves the thin obstacle problem with operator $H$, then $u_k$ solves the thin obstacle problem with operator $H_k$ which from its very definition \eqref{operatore_k} turns out to be $t_k$-rescaled.

In the next lemma, we follow \cite{FeSe17} and construct suitable quadratic functions which act as barriers for $t$-rescaled operators.

\begin{lemma}\label{l:soprasoluzione stretta}
For every $m_0$, $\gamma_0 > 0$, there exist $K$, $t_0$, $C>0$ depending on $n$, $m_0$, $\gamma_0$, $\lambda$ in {\rm \ref{ellip}} and $L$ in \eqref{e.rescaled H}, such that for every 
	\begin{gather*}
		x_0=(x'_0,0) \in B'_{1 / 2}
		\qquad |m| \le m_0
		\qquad 		0 < \gamma \le \gamma_0
		\qquad 0 < t \le t_0
	\end{gather*}
the function
\begin{gather*}
		\eta(x)= t\left(|x'-x'_0|^2 - Kx_{n+1}^2\right) + mx_{n+1} + \gamma
\end{gather*}
satisfies 
\begin{equation}
\max\left\{\|\eta\|_{L^\infty(B_{1/2}(x_0))},\|\nabla\eta\|_{L^\infty(B_{1/2}(x_0))}\right\} \leq 1 + m_0 + \gamma,\label{e.eta2}
\end{equation}
and 
\begin{gather}
H_t (\eta)(x) \ge Ct		\quad\text{and} \quad		H_t (-\eta)(x) \le -Ct
\qquad		\forall \; x \in B_{1/2}(x_0),
\end{gather}
for every $t$-rescaled operator $H_t$.
\end{lemma}

\begin{proof}
We compute
\begin{gather*}
\nabla \eta(x) = (2t(x'-x'_0), -2Kt x_{n+1} +m) \in \R^n \times \R,\\
D^2 \eta(x) = 2t\sum_{i=1}^n e_i \otimes e_i - 2 Kt e_{n+1} \otimes e_{n+1},
	\end{gather*}
and we estimate on $B_{1/2}(x_0)$
\begin{gather*}
\|\eta\|_\infty \leq {1 \over 4}(1 + K)t + \frac{m_0}{2} + \gamma, \qquad
\|\nabla \eta\|_\infty \le (1 + K) t + m_0.
\end{gather*}
Therefore, setting $t_0 = {1 \over 1 + K}>0$, we ensure the validity of \eqref{e.eta2}.
Setting $M = 1 + m_0 + \gamma_0$, we have ($F, F_0$ and their derivatives
are computed in $(x, \eta(x), \nabla \eta(x))$ and $\lambda(M)$ is the function in {\rm \ref{ellip}})
\begin{align*}
H_t \eta &= - \div_x F- \langle \de_z F, \nabla \eta \rangle  - \tr(D_p F \cdot D^2 \eta) + F_0 \\
&\geq - t L(M) - 2t\sum_{i=1}^n \langle D_p F e_i, e_i\rangle + 2 K t \langle D_p F e_{n+1}, e_{n+1}\rangle \ge \\
&\ge -(1+2n)t L(M) + 2Kt \lambda (M) \ge t L(M),
\end{align*}
provided
\begin{gather*}
K \ge (1+n){L(M) \over \lambda(M)}.
\end{gather*}

Similarly, computing the operator for $-\eta$ (hence the arguments of $F,F_0$ and
their derivatives are $(x,-\eta(x), -\nabla \eta(x))$, we get
\begin{align*}
H_t (-\eta) &= - \div_x F + \langle \de_z F, \nabla \eta \rangle  + \tr(D_p F \cdot D^2 \eta) + F_0 \\
&\leq t L(M) + 2t\sum_{i=1}^n \langle D_p F e_i, e_i\rangle - 2 K t \langle D_p F e_{n+1}, e_{n+1}\rangle \le \\
&\leq (1+2n)t L(M) - 2Kt \lambda (M) \le - t L(M).
\end{align*}
\end{proof}

\subsection{Flatness of blowups for thin obstacle problem}

Next, we prove a core result which is crucial in the classification of blowups for the thin obstacle problem: following \cite{FocardiSpadaro20} we show that all blowups around free boundary points need to be flat, by showing that edge-shaped profiles must correspond to points in the interior of the coincidence set.

We recall the notation:
\begin{gather}\label{e.tettuccio}
w(t) = 
\begin{cases}
a^+ t & t\ge 0,\\
a^- t & t\le 0
\end{cases}
\qquad \text{for some $a^+ \le a^-$}.
\end{gather}

\begin{proposition}\label{p.prop di base}
	Suppose $a^+ < a^-$. There exists $\eps = \eps(n, a^+, a^-, \lambda, L)>0$ ($\lambda$ is the ellipticity bound in {\rm \ref{ellip}} and
$L$ is the bound of the rescaled operators in \eqref{e.rescaled H}.
) such that, if $u$ is a solution to the thin obstacle problem in $B_1$ with $\eps$-rescaled operator $H$, such that
	\begin{equation}\label{vicinanza}
		u(x) \le w(x_{n+1}) + \eps \qquad \forall x \in B_1,
	\end{equation}
then $B'_{1 / 2} \subset \Lambda(u)$.
\end{proposition}

\begin{proof}
Let $m_0 = \max \{|a^-|, |a^+|\}$ and $\gamma_0 = 3 + 2 m_0$ in 
Lemma \ref{l:soprasoluzione stretta} and let $K$ and $t_0 > 0$ be the corresponding constants.
Fix $x_0 = (x'_0,0) \in B'_{1 / 2}$, $m = {a^- + a^+ \over 2}$ and $0 < \eps < \min\{1, {1 \over 4}t_0\}$ and define
\begin{gather*}
\eta(x) = 4 \eps (|x' - x'_0|^2 - Kx_{n+1}^2) + mx_{n+1},\\
A = \{s>0 \, : \, u(x) < \eta(x) + s \quad \forall \; x \in \overline{B_{1/2}}(x_0)\}.
\end{gather*}
By \eqref{e.eta2} we have that
	\begin{gather*}
		u(x) - \eta(x) \le w(x_{n+1}) + \eps + 1 + m_0
		\le 2 + 2 m_0 < \gamma_0 \qquad \forall \; x \in B_{1/2}(x_0),
	\end{gather*}
thus implying that $\gamma_0 \in A$. 
Clearly $A$ is a open half line: $A = (\gamma,+\infty)$ for some $\gamma \geq 0$. We want to show that $\gamma = 0$. 

Suppose by contradiction that $0 < \gamma \le \gamma_0$.
By definition $u\leq \eta + \gamma$ and there exists $\overline{x} \in \overline{B_{1/2}}(x_0)$ such that $u(\overline{x})= \eta(\overline{x}) + \gamma$.
If $\eps$ is small enough, we have that
\begin{gather}\label{e.ordine al bordo}
\eta(x)\ge w(x_{n+1}) + \eps \qquad \forall \; x \in \de B_{1/2}(x_0).
\end{gather}
In fact, for $x \in \de B_{1/2}(x_0)$ we have  $|x'-x'_0|^2 = {1 \over 4} - x_{n+1}^2$, so the above inequality reduces to
\begin{gather*}
w(x_{n+1}) \le mx_{n+1} - 4(1+K)\eps x_{n+1}^2 \qquad \mbox{for } |x_{n+1}| \le {1 \over 2}.
\end{gather*}
If $x_{n+1}>0$, this amounts to show
\begin{gather*}
a^+ \le m - 4(1+K)\eps x_{n+1} \qquad \mbox{for } 0 \leq x_{n+1} \leq {1 \over 2}.
\end{gather*}
This is true if $0<\eps\le{a^- - a^+ \over 4(1+K)}$. 
Similarly, this same restriction on $\eps$ implies that a similar inequality holds when $x_{n+1}<0$, i.e.,
\begin{gather*}
a^-\geq m - 4(1+K)\eps x_{n+1} \qquad \mbox{for } - {1 \over 2} \leq  x_{n+1} \leq 0 ,
\end{gather*}
so that we have \eqref{e.ordine al bordo}.
	
Thus, we have
\begin{gather*}
u(x)\le w(x_{n+1}) + \eps \le \eta(x) < \eta(x) +\gamma \quad \forall\;x \in \de B_{1/2}(x_0).
\end{gather*}
This proves that $\overline{x} \notin \de B_{1/2}(x_0)$.
We also notice that, if $\overline{x} \in B'_{1/2}(x_0)$, we would have $u(\overline{x}) = \eta(\overline{x}) + \gamma \ge \gamma >0$. Thus, $\overline{x} \in B_{1/2}(x_0) \setminus \Lambda(u)$.
This means that $Hu=0$ in a neighborhood of $\overline{x}$ and $u$ is touched from above in $\overline{x}$ by the function $\eta + \gamma$, which is a strict supersolution by Lemma \ref{l:soprasoluzione stretta}.
However, by Harnack's inequality (see Corollary \ref{c.harnack}) a solution to $H = 0$ cannot be touched from above by a strict supersolution at an interior point, which leads to a contradiction.
	
Thus, we conclude that $\gamma = 0$: in particular,
\begin{gather*}
0 \le u(x_0) \leq \eta (x_0) + s = s \qquad \forall \; s > 0
\qquad \Longrightarrow \qquad u(x_0) = 0.
\end{gather*}
Since $x_0\in B_{1/2}'$ is arbitrary, we conclude that $B_{1/2}'\subset \Lambda(u)$.
\end{proof}

The main consequence of Proposition \ref{p.prop di base} is that all blowups
of a solution to the thin obstacle problem at a free boundary point are flat, i.e.,
must be of form \eqref{e.tettuccio} with $a^-=a^+$.

\begin{corollary}\label{c.flatness}
Let $u$ be a Lipschitz solution to the thin obstacle problem \eqref{e.var-ineq-c1}
and let $z_0\in \Gamma(u)$.
Then, all blowups of $u$ at $z_0$ are of the form $a x_{n+1}$ for some $a\in \R$.
\end{corollary}

\begin{proof}
Let $u_k$ denote the rescalings at $z_0$: $u_k(z) = u(z_0+t_kx)/t_k$ for some sequence $t_k\downarrow 0$.
Since $\Lip (u_k) \le \Lip(u)$ and $\|u_k\|_{\infty} \le \Lip (u)$, Ascoli-Arzelà's Theorem and Lemma \ref{l:Gamma Conv}, a subsequence of $u_k$ converges uniformly on $B_1$ to a one-dimensional wedge $w$ in \eqref{e.tettuccio} with slopes $a^+ \le a^-$.
	
Suppose by contradiction that $a^+ < a^-$. We consider $\eps > 0$ given by Proposition \ref{p.prop di base}. Choose next $k \geq 1$ big enough to guarantee $t_k < \eps$ and to guarantee that \eqref{vicinanza} holds with $u_k$ in place of $u$. Since $u_k$ is a solution to the thin obstacle problem with operator $H_k$ and the operator $H_k$ is $t_k$-rescaled ($t_k <\eps$), we can apply Proposition \ref{p.prop di base} and infer that $B'_{1/2} \subset \Lambda(u_k)$. But $0 \in \Gamma(u_k)$ by hypotheses, which leads to a contradiction.
So the only possibility is that $a^+ = a^- = a$.
\end{proof}

\subsection{Differentiability at free boundary points}
Proposition \ref{p.prop di base} does not exclude that different subsequences of rescalings produce different blowup limits at the same free boundary points.
This possibility is ruled out by the next result.

\begin{proposition}\label{p.prop da sotto}
	Let $a,m \in \R$ and $m_0 > 0$, $|a|<m_0$ and $|m|\leq m_0$ and $m<a$. There exists $\eps=\eps(n, m_0, a, \lambda, L)>0$ with the following property.
	If $u$ is a solution to the thin obstacle problem in $B_1$ with an $\eps$-rescaled operator $H$ and $u$ satisfies
	\begin{equation*}
		u(x) \ge a x_{n+1} -\eps \qquad \forall x \in B^+_1,
	\end{equation*}
then  $u(x)\geq m x_{n+1}$ for every $x \in B^+_{1/2}$.
\end{proposition}

\begin{proof}
Apply Lemma \ref{l:soprasoluzione stretta} with $m_0$ and $\gamma_0 = 2$, and let $K, t_0 > 0$ be the corresponding constants.
We fix $x_0 = (x'_0,0) \in B'_{1 / 2}$ and $0 < \eps <\{1, {1 \over 4}t_0,{K^{-1}\over 4}\}$ and define
\begin{gather*}
\eta(x) = 4 \eps (|x' - x'_0|^2 - Kx_{n+1}^2) - mx_{n+1}\\
A=\left\{s>0 \, : \, u(x) > -\eta(x) - s \quad \forall x \in \overline{B^+_{1/2}}(x_0)\right\}.
\end{gather*}
We have that 
\begin{gather*}
u(x) + \eta(x) \geq a x_{n+1} - \eps -1 - m x_{n+1}
> - 2 = - \gamma_0 \qquad \forall \; x \in \overline{B_{1/2}^+}(x_0),
\end{gather*}
where we used that $a>m$. 
Therefore, $\gamma_0 \in A$; in particular, $A$ is not empty and has the form $A = (\gamma,+\infty)$ for some $\gamma \geq 0$.
We want to show that $\gamma = 0$. 
Suppose by contradiction that $0 < \gamma \le \gamma_0$. 
Arguing as for Proposition \ref{p.prop di base}
we infer that $u\geq -\eta -\gamma$ in $B^+_{1/2}(x_0)$ and 
there exists $\overline{x} \in \overline{B^+_{1/2}}(x_0)$ such that $u(\overline{x})= -\eta(\overline{x}) -\gamma$.
Note that for $\eps$ small enough we get
\begin{gather*}
\eta(x)\ge - a x_{n+1} + \eps \qquad \forall x \in (\de B_{1/2}(x_0))^+.
\end{gather*}
Indeed, for $x \in (\de B_{1/2}(x_0))^+$ we have  $|x'-x'_0|^2 = {1 \over 4} - x_{n+1}^2$, so the above inequality reduces to
\begin{gather*}
a \ge m + 4(1+K)\eps x_{n+1} \qquad \forall \; x_{n+1} \in \left(0,{1 / 2}\right],
\end{gather*}
which is true if $0<\eps\le{a-m_0 \over 2(1+K)}\leq {a-m \over 2(1+K)}$. Hence, under this assumption we conclude that
\begin{gather*}
u(x)\geq a x_{n+1}-\eps \ge -\eta(x) > -\eta(x) -\gamma \qquad \forall\;
x \in (\de B_{1/2}(x_0))^+.
\end{gather*}
We deduce that $\overline{x} \notin (\de B_{1/2}(x_0))^+$.
Moreover, $\overline{x}\not\in B_{1/2}'(x_0)$, because in this case 
$u(\overline{x}) = -\eta(\overline{x}) - \gamma \le -\gamma <0$, against the unilateral constraint. 
Thus, $\overline{x} \in B^+_{1/2}(x_0)$, which means that $Hu=0$ around $\overline{x}$. 
We reach then a contradiction by noticing that the solution $u$ is touched from below by the function $-\eta - \gamma$ which is a strict subsolution in a neighborhood of $\overline{x}$ (see Corollary \ref{c.harnack}).

Thus, we conclude that $\gamma = 0$, so that for every $x \in B^+_{1/2}(x_0)$ we have $u(x) \ge -\eta(x)$.
In particular, for every $0\le x_{n+1}\le {1 \over 2}$,
\begin{gather*}
u(x'_0,x_{n+1}) \ge -\eta(x'_0,x_{n+1})=4\eps K x_{n+1}^2 +mx_{n+1} \geq m x_{n+1}.
\end{gather*}
This conclusion being true for every $x_0'\in B'_{1/2}$ and every $0\leq x_{n+1}\leq 1/2$, we conclude the proof of the proposition.
\end{proof}

Clearly, an analogue statement of Proposition \ref{p.prop da sotto} holds in $B^-_1$.
The main consequence of the previous result is the uniqueness of blowups at any free boundary point both for the thin and the boundary obstacle problem.

\begin{proposition}\label{p.uniqueness}
Let $u$ be a Lipschitz solution to either the thin or to the boundary obstacle problem \eqref{e.var-ineq-c1}.
Then, for every $z_0\in \Gamma(u)$ there exists $a_{z_0}\in \R$ such that
the linear function $u_{z_0}(x) = a_{z_0} x_{n+1}$ is the unique blowup limit at $z_0$, i.e.
\[
u_t(x) = \frac{u(z_0+tx)}{t} \to u_{z_0}(x) \qquad \textup{uniformly in $B_1$, as $t\to 0$.}
\]
In particular, by taking $x = e_{n+1}$, we have that $u$ is differentiable at $z_0$ and $\nabla u(z_0) = (0,a_{z_0})\in \R^n\times\R$.
\end{proposition}

\begin{proof}
Suppose by contradiction that there are two different sequences $t^{(i)}_k \downarrow 0$, $i=1,2$ such that the limit of the rescalings 
\[
u^{(i)}_k(x) = \frac{u(z_0+t^{(i)}_kx)}{t^{(i)}_k}
\]
are the functions $a^{(i)} x_{n+1}$ with $a^{(2)}<a^{(1)}$.
We choose $m$ such that $a^{(2)}<m<a^{(1)}$.
Since the rescalings $u^{(1)}_k$ solve the thin obstacle problem with $t^{(1)}_k$-rescaled operators, for $k$ big enough we have that the hypotheses of 
Proposition \ref{p.prop da sotto} are satisfied (with parameters $m_0= \Lip(u)$,
$a = a^{(1)}$, $m$) and therefore we get that $u^{(1)}_k(x)\geq m^{}x_{n+1}$ for every $x \in B^+_{1/2}$ for $k$ large enough, which means that
\begin{gather*}
u(x) \geq m x_{n+1} \qquad \forall x \in B^+_{s}(z_0),
\end{gather*}
for a suitable $s>0$.
This is a contradiction to the fact that 
\[
m x_{n+1} \leq u^{(2)}_k(x) = \frac{u(z_0+t^{(2)}_kx)}{t^{(2)}_k} \to a^{(2)}x_{n+1}<m x_{n+1} \quad \forall \;x_{n+1}> 0.
\]
\end{proof}

\subsection{On value of the normal derivative on the free boundary}
We show that the gradient of the solutions at free boundary points is prescribed by the Signorini boundary condition $F_{n+1}(x_0,0, \nabla u(x_0))=0$.

To this aim we start with the following lemma.

\begin{lemma}\label{l.prop di base - BOP}
	Suppose $m > a$. There exists $\eps = \eps(n, a, m, \lambda, L)>0$ such that, if $u$ is a solution to the boundary obstacle problem in $B_1^+$ with $\eps$-rescaled operator $H$, such that
	\begin{gather*}
		u(x) \le ax_{n+1} + \eps \qquad \forall \; x \in B_1^+, \\
		\de_{n+1} u(x) > m  \qquad \forall \; x \in B_1' \setminus \Lambda(u),
	\end{gather*}
then $B'_{1 / 2} \subset \Lambda(u)$.
\end{lemma}

\begin{proof}
	Let $m_0 = \max \{|a|, |m|\}$ and $\gamma_0 = 3 + 2 m_0$ in 
	Lemma \ref{l:soprasoluzione stretta} and let $K$ and $t_0 > 0$ be the corresponding constants.
	Fix $x_0 = (x'_0,0) \in B'_{1 / 2}$ and $0 < \eps \le \min\{1, {1 \over 4}t_0\}$ and define
	\begin{gather*}
		\eta(x) = 4 \eps (|x' - x'_0|^2 - Kx_{n+1}^2) + mx_{n+1},\\
		A = \{s>0 \, : \, u(x) < \eta(x) + s \quad \forall \; x \in \overline{B_{1/2}^+(x_0)}\}.
	\end{gather*}
It is immediate to verify that $\gamma_0\in A$ because we have that
	\begin{gather*}
		u(x) - \eta(x) \le a x_{n+1} + \eps + 1 + m_0
		\le 2 + 2 m_0 < \gamma_0 \qquad \forall \; x \in \overline{B_{1/2}^+}(x_0).
	\end{gather*}
We show that $\gamma:=\inf A=0$.
Suppose by contradiction that $0 < \gamma \le \gamma_0$.
	By definition, $u\leq \eta + \gamma$ and there exists $\bar x \in \overline{B_{1/2}^+}(x_0)$ such that $u(\bar x)= \eta(\bar x) + \gamma$.
	If $\eps<\frac{m-a}{2(1+K)}$, then it is simple to verify that
\begin{gather}
\eta(x)\ge a x_{n+1} + \eps \qquad \forall \; x \in (\de B_{1/2}(x_0))^+.
\end{gather}
Therefore $\bar x \notin (\de B_{1/2}(x_0))^+$.
On the other hand, if $\bar x \in B'_{1/2}(x_0)$, then $u(\bar x) = \eta(\bar x) + \gamma \ge \gamma >0$. Thus, $\bar x \in B_1' \setminus \Lambda(u)$: i.e.,
	\begin{gather*}
		u(x) \le \eta(x) + \gamma
		\qquad
		\forall \; x \in \overline{B_{1/2}^+(x_0)},
		\qquad \qquad
		u(\bar x) = \eta(\bar x) + \gamma.
	\end{gather*}
It follows then that necessarily $\de_{n+1} u(\bar x) \le \de_{n+1} \eta(\bar x) = m$, which contradicts our hypotheses.
Finally, if $\bar x \in B_{1/2}^+(x_0)$, than we contradicts Harnack's inequality (see Corollary \ref{c.harnack}) because $Hu=0$ in a neighborhood of $\bar x$ and $u$ is touched from above by a strict supersolution $\eta + \gamma$.
	
Thus, we conclude that $\gamma = 0$ and therefore $u(x_0)=0$ for all $x_0\in B_{1/2}'$.
\end{proof}

As a consequence we deduce that the co-normal derivative must vanish at free boundary points.

\begin{proposition}\label{p.uniqueness Signorini}
	Let $u$ be a Lipschitz solution to the boundary obstacle problem \eqref{e.var-ineq-c1}.
	Then, for every $z_0\in \Gamma(u)$ we have that $F_{n+1}(z_0, 0, 0, \de_{n+1}u(z_0)) = 0$.
\end{proposition}
\begin{proof}
	Let $a = \de_{n+1} u(z_0)$. By Lemma \ref{l:Gamma Conv} we know that $F_{n+1}(z_0, 0, 0, a)\le 0$.
	Assume by contradiction that $\tau := {1 \over 2} F_{n+1}(z_0, 0, 0, a) < 0$ and fix any constant $m>a$ such that 
	$F_{n+1}(z_0, 0, 0, m) < \tau$. We can find such $m$ since $F_{n+1}(z_0, 0, 0, \cdot)$ is strictly monotone increasing and continuous. 

We consider $\eps > 0$ given by Lemma \ref{l.prop di base - BOP}. We want to apply that result to $u_k(x) = k u(z_0 + x / k)$ for $k \ge 1$ big enough.
To this aim, given $x \in B_1' \setminus \Lambda(u_k)$, we note that $z_0 + x / k \in B_1' \setminus \Lambda(u)$. Since $u$ is $C^1$ around $z_0 + x / k$, by the Signorini boundary conditions we have that
	\begin{gather*}
		F_{n+1}(z_0 + x / k, u_k(x) / k, \nabla u_k(x))=0
	\end{gather*}
Therefore, for every $x \in B_1' \setminus \Lambda(u_k)$ we have that
\begin{align*}
 - F_{n+1}(z_0, 0, 0, \de_{n+1} u_k(x)) &\\
&		= F_{n+1}(z_0 + x / k, u_k(x) / k, \nabla u_k(x))
		- F_{n+1}(z_0, 0, 0, \de_{n+1} u_k(x)) \\
		& \le \|\nabla F_{n+1}\|_\infty
		\left({1 + \|u_k\|_\infty \over k} + \|\nabla' u_k\|_\infty \right)
		=o(1) \qquad \mbox{\rm for } k \to \infty,
	\end{align*}
where we use Frehse's Theorem \ref{Frehse} to deduce that 	\begin{gather*}
		|\nabla' u_k(x)| = |\nabla' u(z_0 + x / k) - \nabla' u(z_0)| \le \omega(1/k),
	\end{gather*}
and $\sup_k\|u_k\|_\infty <\infty$.
We can therefore choose $k$ big enough to ensure that $\tau \le -\omega(1/k)$. So that
\begin{align*}
		F_{n+1}(z_0, 0, 0, m)
		< \tau \le -\omega(1/k)
		\le F_{n+1}(z_0, 0, 0, \de_{n+1} u_k(x))
		\; \; \Longrightarrow \; \;
		m < \de_{n+1} u_k(x).
	\end{align*}
Thus we have proven that $m < \de_{n+1} u_k(x)$ for every $x \in B_1' \setminus \Lambda(u_k)$, if $k$ is big enough. Moreover, 
since $\nabla u(z_0) = (0,a)$, we have that
$u_k(x) \le ax_{n+1} + \eps$ for every $x \in B_1^+$ and large enough $k$. Therefore, since $u_k$ is a solution to the boundary obstacle problem with $\eps$-rescaled operator $H_k$ (if $k^{-1}<\eps$), we can apply Lemma \ref{l.prop di base - BOP} to get that $B_{1/2}' \subset \Lambda(u_k)$, against the assumption that $z_0 \in \Gamma(u)$.
\end{proof}

\subsection{Continuity of the normal derivative}
Building upon the previous results, we are ready to prove the continuity of the derivatives stated in Theorem \ref{t.c1}.

We start with the following proposition.

\begin{proposition}\label{p.blowup}
	Let $u$ be a Lipschitz solution to either the thin or the boundary 
	obstacle problem \eqref{e.var-ineq-c1}.
	If $\{z_k\} \subset \Gamma(u)$, $\{t_k\} \subset \R$ such that $t_k \to 0$, $z_k \to z_0 \in \Gamma(u)$, $0<t_k < 1 - |z_k|$, then 
	\[
	u_k(x) = {u(z_k + t_k x) \over t_k} \to \de_{n+1}u(z_0) x_{n+1} \qquad\textup{uniformly on $B_1$}.
	\]
\end{proposition}

\begin{proof}
	By Lemma \ref{l:Gamma Conv} and Corollary \ref{c.flatness}, up to a subsequence we have that $u_k(x) \to a x_{n+1}$ uniformly on $B_1$, for some $a \in \R$.
	We consider separately the two obstacle problems.
	
	\smallskip
	
	\textit{Thin obstacles}.
	For every $\delta>0$ we define
	\begin{gather*}
		w_\delta(x_{n+1}) =	\begin{cases}
			(\de_{n+1}u(z_0) -\delta) x_{n+1} & x_{n+1}\ge 0, \\
			(\de_{n+1}u(z_0) +\delta) x_{n+1} & x_{n+1}\le 0.
		\end{cases}
	\end{gather*}
	We then consider $\eps>0$ in Proposition \ref{p.prop da sotto} with $\de_{n+1}u(z_0)$ in place of $a$ and $\de_{n+1}u(z_0)-\delta$ in place of $m$.
	By the uniqueness of blowups in Proposition \ref{p.uniqueness}, there exists $0 < t_\delta < \min\{1-|z_0|, \eps\}$ such that
	\begin{gather*}
		\left| u_{t}(x) - \de_{n+1}u(z_0)x_{n+1} \right| \le \eps \qquad \forall x \in B_1,\qquad \forall\; t<t_\delta,
	\end{gather*}
	where
	\begin{gather*}
		u_{t}(x)={u(z_0 + t x) \over t}.
	\end{gather*}
	By applying Proposition \ref{p.prop da sotto} to both sides of the ball, for small values of $t$ we get that $u_{t}(x) \ge w_\delta(x_{n+1})$ for every $x \in B_{1/2}$. So there exists a small radius $r_\delta>0$ such that $u(x) \ge w_\delta(x_{n+1})$ for every $x \in B_{r_\delta}(z_0)$.
	
	Now let $z_k$, $t_k$ and $u_k$ as in the statement. We want to show that $a = \de_{n+1}u(z_0)$. For $k\ge1$ big enough, indeed, we have $|z_k-z_0|\le r_\delta/2$ and $t_k \le r_\delta/2$, so that $B_{t_k}(z_k) \subset B_{r_\delta}(z_0)$. Thus, $u_k(x) \ge w_\delta(x_{n+1})$ for every $x \in B_1$. Letting $k\to+\infty$, we then get $|a-\de_{n+1}u(z_0)|\le \delta$. Since $\delta$ was arbitrary small, the proof is complete.
	
	\smallskip
	
	\textit{Boundary obstacles}.
We show that $a = \de_{n+1} u(z_0)$. Indeed, if $a > \de_{n+1} u(z_0)$, then by Proposition \ref{p.uniqueness Signorini} we have that
	\begin{gather*}
		F_{n+1}(z_0, 0, 0, a) > F_{n+1}(z_0, 0, 0, \de_{n+1} u(z_0)) = 0,
	\end{gather*}
which is a contradiction to Lemma \ref{l:Gamma Conv}.
If instead $a < \de_{n+1} u(z_0)$, let
	\begin{gather*}
		\hat u_k (x) = k u(z_0 + x / k).
	\end{gather*}
By Proposition \ref{p.uniqueness} we know that $\hat u_k(x) \to \de_{n+1} u(z_0) x_{n+1}$ uniformly and 
by Proposition \ref{p.prop da sotto} for every $a < m < \de_{n+1} u(z_0)$ we have that $\hat u_k(x) \ge m x_{n+1}$ for every $x \in B_{1/2}^+$ if $k$ is big enough. So there exists a small radius $r > 0$ such that $u(x) \ge m x_{n+1}$ for every $x \in B_r^+(z_0)$.
If $k$ is big enough, then $|z_k-z_0|\le r/2$ and $t_k \le r/2$, so that $B^+_{t_k}(z_k) \subset B^+_r(z_0)$. Thus, $u_k(x) \ge m x_{n+1}$ for every $x \in B^+_1$. Letting $k \to + \infty$, we then get $a \ge m$, which is a contradiction.
\end{proof}

\subsection*{Proof of Theorem \ref{t.c1}}
We consider separately the two obstacle problems.

\smallskip

\textit{Thin obstacles.}
We start observing that $u$ is $C^{1,\alpha}$ regular around points of $B^+_1 \cup B'_1 \setminus \Lambda(u)$ for every $0<\alpha<1$, since it solves a quasi-linear elliptic equation with $C^1$-regular operator. If $z_0$ is an interior point of $\Lambda(u)$ with respect to the relative topology of $B'_1$, then we can find $r>0$ small such that $B'_r(z_0) \subset \Lambda(u)$. In this case, $u$ is a solution to the Dirichlet problem in $B^+_r(z_0)$ with null boundary datum on the flat portion of the half-ball. Due to a result of Giaquinta and Giusti \cite{GiaquintaGiusti84} (see Appendix \ref{a.regolarità}), $u$ is then $C^{1,\alpha}$ around $z_0$, for every $0<\alpha<1$. It is left to prove that $u$ is $C^1$ around points of $\Gamma(u)$.

Let $z_0 \in \Gamma(u)$ and $\{y_k\}_{k \ge 1} \subset B^+_1 \cup B'_1$ be such that $y_k \to z_0$. Without loss of generality, we may assume that the whole sequence $\{y_k\}_{k \ge 1}$ is contained either in $\Gamma(u)$ or outside $\Gamma(u)$.
	
\smallskip 
	
\textit{Case $\{y_k\}_{k \ge 1} \subset B^+_1 \cup B'_1 \setminus \Gamma(u)$.} 
For every $k \ge 1$ we choose $z_k \in \Gamma(u)$ such that $t_k := \dist(y_k, \Gamma(u)) = |z_k - y_k|>0$. Set
\begin{gather*}
	\tau_k = 2t_k, \quad p_k = {y_k - z_k \over \tau_k}, \quad B_k = B_{t_k}(y_k), \quad C_k = B^+_{t_k / 2}(y_k), \quad D=B^+_{1 \over 8}(p).
\end{gather*}

Without loss of generality and upon extracting a subsequence, we may assume
\begin{gather*}
	\tau_k < 1-|z_k|, \quad p_k \to p\in \de B_{1 \over 2}, \quad |p_k-p|<{1 \over 8}.
\end{gather*}

Next, we set
\begin{gather*}
	u_k(x) = {u(z_k + \tau_k x) \over \tau_k} \qquad \forall x \in B_1.
\end{gather*}

For every fixed $k \ge 1$, either $B_k \cap \Lambda(u) = \emptyset$ or $B'_k \subset \Lambda(u)$ (depending on whether or not $y_k$ belongs to the interior of the coincidence set of $u$).

In both cases (using either De Giorgi's Theorem \cite{DeGiorgi57} or the already quoted result of Giaquinta and Giusti \cite{GiaquintaGiusti84}, see Appendix \ref{a.regolarità}), we conclude that $u\in C^{1,\alpha}(C_k)$ with uniform bounds, so that there exists a constant $C>0$ such that
\begin{gather*}
	|\de_{n+1} u_k(x) - \de_{n+1} u_k(y)|
	\le C \tau_k^\alpha |x-y|^\alpha \quad \forall x,y \in D.
\end{gather*}

Considering the uniform boundedness $\|\de_{n+1} u_k \|_{L^\infty(D)} \leq \|\de_{n+1} u \|_{L^\infty(B_1)}$, we can apply Ascoli-Arzelà's Theorem to deduce that $\de_{n+1} u_k$ converges (upon extracting a subsequence) uniformly in $D$.
Moreover, since by Proposition \ref{p.blowup}, we have that $u_k(x) \to \de_{n+1}u(z_0) x_{n+1}$ uniformly on $B_1$, necessarily it must hold that 
$\de_{n+1} u_k\to \de_{n+1}u(z_0)$.

We can repeat the same argument for the negative part of the balls to get that $\de_{n+1} u_k \to \de_{n+1}u(z_0)$ (up to further subsequences) in the whole ball $B_{1 \over 8}(p)$. Moreover, since the limit is independent of the subsequence, the entire sequence $u_k$ satisfies the same conclusion.
In particular, $|p_k - p|<{1 \over 8}$ implies that
\begin{gather*}
	|\de_{n+1} u_k (p_k) - \de_{n+1}u(z_0)| \le \|\de_{n+1} u_k - \de_{n+1}u(z_0)\|_\infty \to 0,
\end{gather*}
thus proving that $\de_{n+1} u(y_k) = \de_{n+1} u (z_k + \tau_k p_k) = \de_{n+1} u_k (p_k) \to \de_{n+1}u(z_0)$.

\medskip

\textit{Case $\{y_k\}_{k \ge 1} \subset \Gamma(u)$.}
By Proposition \ref{p.uniqueness} we have that $u$ is differentiable at $y_k$, hence there exists $0<t_k<1-|y_k|$ such that
\begin{gather*}
	\left| {u(y_k + t_k e_{n+1}) \over t_k} - \de_{n+1} u(y_k) \right| \le {1 \over k}.
\end{gather*}
However by Proposition \ref{p.blowup} we have that
\begin{gather*}
	{u(y_k + t_k e_{n+1}) \over t_k} \to \de_{n+1}u(z_0).
\end{gather*}
We conclude that $\de_{n+1} u(y_k) \to \de_{n+1} u(z_0)$,
thus completing the proof of the continuity of the normal derivative of the solution to the thin obstacle problems.

\medskip

\textit{Boundary obstacles}.
We start noticing that at any $z_0\in \Gamma(u)$ by the ellipticity hypothesis \ref{ellip} the function
\begin{gather*}
	\phi(t)= F_{n+1}(z_0,0,te_{n+1})
\end{gather*}
is monotone increasing and $\phi'(t) = \partial_{p_{n+1}}F_{n+1}(z_0,0,te_{n+1})\geq \lambda(t)$, the constant $\lambda$ being uniformly positive for $t$ in any compact set.
Moreover, using the result by Lieberman \cite{Lieberman88} for the regularity of Neumann's problem (see Appendix \ref{a.regolarità}), we have that $u$ is $C^{1,\alpha}(B_1^+\cup B_1'\setminus \Lambda(u))$ and it follows from the variational inequality \ref{e.var-ineq-c1} that
\begin{gather*}
	F_{n+1}((x',0), u(x',0), \nabla u (x',0)) = 0\qquad \forall\;(x',0)\in B_1'\setminus \Lambda(u).
\end{gather*}
Therefore, if $y_k \in B_1'\setminus\Lambda(u)$ with $y_k\to z_0$, then $\nabla' u(y_k)\to 0$ by Frehse result and
\[
F_{n+1}(y_k,u(y_k),\nabla u(y_k))=0;
\]
hence, for any converging subsequence $\nabla u(y_{k_j})\to (0,a)\in \R^n\times \R$, we have that
\[
F_{n+1}(z_0,0,(0,a))=0.
\]
By the strict monotonicity of $\phi(t)= F_{n+1}(z_0,0,te_{n+1})$ 
we must have $\partial_{n+1}u(z_0)=a$, i.e., $\partial_{n+1}u(y_{k_j})\to 0$.
On the other hand, if $y_k \in \Lambda(u)$ with $y_k\to z_0$, we can argue as for the thin obstacle problem, inferring that $\de_{n+1} u(y_k) \to \de_{n+1} u(z_0)$.
This proves the continuity of the normal derivative at any free boundary point.
\qed

\section{$C^{1,\alpha}$ Regularity}

In this section we prove the $C^{1,\alpha}$ regularity for the solutions to the variational inequalities with boundary obstacles \eqref{e.var-ineq-c1}, which we
recall here for readers' convenience:
\begin{gather}\label{e.var-ineq-c2}
\int_{B_1^+} \langle F(x,u,\nabla u),\nabla v - \nabla u \rangle
+ F_0(x,u,\nabla u) (v -u) \ge 0 \quad \forall \, v \in \mathcal{A}_g,\\
\mathcal{A}_g := \left\{v\in W^{1,\infty}(B_1^+) \;:\; v\vert_{\de B_1^+\setminus B_1'}= g,  v\vert_{B_1'} \ge 0\right\},
\qquad u\in \mathcal{A}_g.\notag
\end{gather}

We will prove the following result.

\begin{theorem}\label{t.c2}
Let $u$ be a Lipschitz solution to the variational inequality \eqref{e.var-ineq-c2} for the boundary obstacle problem.
Then, there exists $\alpha\in(0,1)$ such that $u \in C^{1,\alpha}(B^+_1 \cup B'_1)$.
\end{theorem}

Clearly, Theorem \ref{t.1} (ii) is a corollary of Theorem \ref{t.c2} by following the usual local straightening of the boundary described in Section \ref{s.rettificare}.

We will prove Theorem \ref{t.c2} by extending to the present nonlinear case the techniques developed by Uralt'seva \cite{Uraltseva85, Uraltseva89} based on De Giorgi's method.

\subsection{Caccioppoli inequality for the tangential Derivatives} \label{sec:tangential-derivatives}

We prove Caccioppoli-type inequalities for the tangential derivatives of $u$, namely $\pm \de_i u$, $i = 1,\dots,n$. Before moving on with the proof, we show a simple lemma (see also \cite{Frehse77}).
Here, we introduce the following notation for the difference quotients:
\begin{gather*}
D^h_i w(x) = {w(x + he_i) - w(x) \over h},
\end{gather*}
whenever the above expression makes sense, i.e. for every $i = 1, \dots, n$, $h \ne 0$ and $x \in B_1^+ \cup B_1'$ such that $x + h e_i \in B_1^+ \cup B_1'$.

\begin{lemma}\label{l:rappo}
Let $u \in \mathcal{A}_g$ and $\varphi \in W^{1,\infty}(B_1^+)$ such that $\supp \varphi \subset B^+_r\cup B'_r(x_0)$, where $x_0 \in B_1'$ and $0 < r < 1 - |x_0|$. Then, for every $0 < h < {1 - |x_0| - r \over 2}$ and $k \geq 0$, there exists $\eps_0 = \eps_0 (h,\|\varphi\|_\infty)>0$ such that
\begin{equation}\label{test Frehse}
v := u + \eps D^{-h}_i (\varphi^2 (D^h_i u - k)_+)\in 
\mathcal{A}_g\quad \forall\; i=1,\dots,n \quad \forall\;\eps\in(0,\eps_0).
\end{equation}
\end{lemma}

\begin{proof}
If $x \in B_1^+ \cup B_1'$ with $|x-x_0| \ge r + h$, then $\varphi(x) = \varphi(x-he_i) = 0$, so that $v(x) = u(x)$; in particular, $v\vert_{\de B_1^+\setminus B_1'}= g$ and $v(x) \geq 0$ for every $x\in B_1'\setminus B_{r+h}(x_0)$.
Therefore, we need only to show that $v(x) \geq 0$ for every $x\in B_{r+h}'(x_0)$. Note that 
\begin{align*}
v(x) &= u(x) + {\eps \over h}\left( \varphi^2(x) (D^h_i u(x) - k)_+ - \varphi^2(x - he_i) \left( {u(x) \over h}  - {u(x - he_i) \over h} -k\right)_+ \right) \\
&\geq u(x) - {\eps \over h} \varphi^2(x - he_i) \left( {u(x) \over h} - {u(x - he_i) \over h} -k\right)_+\\
&\geq u(x) - \eps \,{\|\varphi\|_\infty^2 \over h^2} \left( {u(x)} - k h\right)_+.
\end{align*}
Thus $v(x) \ge u(x) \ge 0$ if $u(x) \le kh$ and $v(x) \ge (1 - {\|\varphi\|_\infty^2 \over h^2} \eps) u(x) + {\|\varphi\|_\infty^2 \over h}k\eps$ otherwise. Therefore, if $\eps$ is sufficiently small, then $v(x) \geq 0$ in both cases.
\end{proof}

In the next proposition we will make use of the previous lemma to show that $\de_i u$ satisfy a Caccioppoli inequality.

\begin{proposition}\label{prop Caccioppoli w}
Let $u$ be the solution to the boundary obstacle problem. There exists $c = c(n,M)>0$ such that the functions $w = \pm \de_i u$, $i = 1,\dots,n$ satisfy
\begin{gather}\label{Caccioppoli w}
\int_{A(k,r)} |\nabla w|^2
\leq {c \over (R-r)^2} \int_{A(k,R)} (w-k)^2+ c\, |A(k,R)|,
\end{gather}
for every $k \geq 0$, $x_0 \in B_1'$, and $0<r<R<1 - |x_0|$,
where $A(k,s) = \{w \ge k\} \cap B_s^+(x_0)$ and 
$|E|$ denotes the Lebesgue measure of a set $E$ in $\R^{n+1}$.
\end{proposition}

\begin{proof}
Let $i = 1,\dots,n$ and $\varphi \in C^\infty(B_1^+(x_0))$ such that $\varphi \equiv 1$ on $B_r^+(x_0)$, $\varphi \equiv 0$ outside $B_R^+(x_0)$ and $|\nabla \varphi| \le {c \over R-r}$. We plug $v$ as in \eqref{test Frehse} into \eqref{e.var-ineq-c2} to get
\begin{gather*}
\int_{B_1^+} \langle D^{h}_i (F(x,u,\nabla u)),  \nabla \zeta_h \rangle - F_0(x,u,\nabla u) D^{-h}_i \zeta_h \le 0,
\end{gather*}
where $\zeta_h = \varphi^2 (D^h_i u - k)_+$. We let $h \to 0^+$ to infer that
\begin{equation*}
\int_{B_1^+} \langle a \nabla \de_i u - q, \nabla \zeta \rangle \le 0
\end{equation*}
with $\zeta := \varphi^2 (\de_i u - k)_+$ and
\begin{gather*}
a(x) := D_pF(x,u(x),\nabla u(x))\\
q(x) := - \de_{x_i} F(x,u(x),\nabla u(x)) - \de_z F(x,u(x),\nabla u(x)) \de_i u(x) + F_0(x,u(x),\nabla u(x)) \; e_i.
\end{gather*}
We notice that $\|a\|_\infty + \|q\|_\infty \le C(\|u\|_\infty, \Lip(u))$ and $\langle a(x)\xi,\xi \rangle \ge \lambda(\|u\|_\infty, \Lip(u)) |\xi|^2$.
Standard calculations then lead to \eqref{Caccioppoli w} for $w = \de_i u$. The case $w = - \de_i u$ is analogous.
\end{proof}

\subsection{Normal Derivative}
Now we will deal with the normal derivative of $u$.

\begin{proposition}
Let $u$ be the solution to the boundary obstacle problem. There exists $c = c(n,\|u\|_\infty,\Lip(u))>0$ such that the functions $w(x) = \pm F_{n+1}(x,0,0,\de_{n+1} u(x))$ satisfy
\begin{equation}\label{Caccioppoli n+1}
\int_{A(k,r)} |\nabla w|^2
\le {c \over (R-r)^2} \int_{A(k,R)} (w-k)^2+ c\, |A(k,R)|,
\end{equation}
for every $k \geq 0$, $x_0 \in B_1'$, and $0<r<R<1 - |x_0|$.	
\end{proposition}

\begin{proof}
The case $w(x) = F_{n+1}(x,0,0,\de_{n+1} u(x))$ is straightforward since in this case $\{w > k\}\cap B_1'=\emptyset$ and an elliptic differential equation is satisfied by $u$ in $B_1^+$. 
Thus, we focus on the case $w(x) = - F_{n+1}(x,0,0,\de_{n+1} u(x))$. We divide the proof into steps.

\medskip

\textbf{Step 1.} 
We can reduce to the case 
\begin{gather}\label{ip extra F gratis}
\de_{p_{n+1}} F'(x,u(x),\nabla u(x)) = 0\qquad \forall \;x \in \Lambda(u),
\end{gather}
where we write $F=(F', F_{n+1})\in\R^n\times \R$.
To this aim, we make a change of variables
\begin{gather*}
y = \Phi(x',x_{n+1}) = \left(x' + b(x', x_{n+1}), x_{n+1}\right),
\end{gather*}
with $b \in C^1(\R^n \times \R, \R^n)$ given by Whitney's $C^1$-extension Theorem applied to the functions $b_i:\overline{B_{1/2}'}\to \R$ and $d_i:\overline{B_{1/2}'}\to \R^{n+1}$ for $i=1, \ldots, n$,
\begin{gather}\label{propietà di b}
\begin{cases}
b_i(x',0) = 0, \\
d_i(x',0) = (0,V_i(x'))\in \R^n\times\R,
\end{cases}
\quad \forall\;(x',0)\in \overline{B_{1/2}'}
\end{gather}
where
\begin{gather}
V_i(x') = {\de_{p_{n+1}} F_i(x',0,0,0,\de_{n+1} u(x',0)) \over \de_{p_{n+1}} F_{n+1}(x',0,0,0,\de_{n+1} u(x',0))}.
\end{gather}
We remark that $\de_{p_{n+1}} F_{n+1} > 0$ thanks to {\rm \ref{ellip}} and by 
Theorem \ref{t.1} $V_i$ is continuous, so that we are in position to apply Whitney's Theorem and get functions $b_i$ such that $\nabla b_i(x',0) = (0,V_i(x'))$ for all $x'\in \overline{B_{1/2}'}$.

By definition $\Phi$ is a local $C^1$-diffeomorphism between $B_{1/2}$ and a neighborhood of the origin, such that $\Phi\vert_{B_{1/2}'} = \textup{Id}$ 
and $\bar u(x) = u(\Phi(x))$ solves a variational inequality
\begin{gather*}
\int_{B_r^+}\langle \bar F(x,\bar u,\nabla \bar u),\nabla \bar v - \nabla \bar u \rangle + \bar F_0(x,\bar u,\nabla \bar u) (\bar v - \bar u) \, dx,\\
\qquad \forall\; \bar v\vert_{B_r'}\geq 0, \;
\bar v\vert_{(\partial B_r)^+} = \bar u\vert_{(\partial B_r)^+},
\end{gather*}
for suitable $r>0$ (depending on the diffeomorphism $\Phi$) and
\begin{gather*}
\bar F(x,z,p) = |\det A(x)|\,A(x)^{-1}F(\Phi(x),z,(A(x)^{-1})^Tp),\\
\bar F_0(x,z,p) = |\det A(x)|\, F_0(\Phi(x),z,(A(x)^{-1})^Tp),\\
A(x) = D\Phi(x)  =
\begin{pmatrix}
\mbox{Id}_n + D' b(x)	&	 \de_{n+1} b(x) \\
0	&	1
\end{pmatrix}.
\end{gather*}
By direct calculations, we have that
\begin{gather*}
\bar F_i(x',0,0,0,p_{n+1}) = F_i(x',0,0,0,p_{n+1}) - F_{n+1}(x',0,0,0,p_{n+1}) V_i(x').
\end{gather*}
Differentiating with respect to the $p_{n+1}$ variable, we get
\begin{gather*}
\de_{p_{n+1}} \bar F_i(x',0,0,0,p_{n+1}) = \de_{p_{n+1}} F_i(x',0,0,0,p_{n+1}) - \de_{p_{n+1}} F_{n+1}(x',0,0,0,p_{n+1}) V_i(x').
\end{gather*}
Since for $(x',0)\in \Lambda(u)$ we have that $\de_{n+1} \bar u(x',0) = \de_{n+1} u(x',0)$, setting $p_{n+1}=\de_{n+1} \bar u(x',0)$
we get
\begin{gather*}
\de_{p_{n+1}} \bar F_i(x',0,0,0,\de_{n+1} \bar u(x',0)) = 0.
\end{gather*}
Therefore, up to applying the local diffeomorphism $\Phi$, we can always assume that 
\eqref{ip extra F gratis} holds.

\medskip

\textbf{Step 2.} Let $\zeta \in C^\infty(B_1^+ \cup B_1')$ be such that 
\begin{equation}\label{e:supporto zeta}
\supp \zeta \;\cap\; (\de B_1)^+ = \emptyset, \qquad
\supp \zeta \;\cap\; B_1' \subset \subset \Lambda(u).
\end{equation}
Then, 
\begin{align*}
 \int_{B_1'} F_{n+1}(x,u,\nabla u) \; \de_{n+1} \zeta& = - \int_{B_1^+} \div (F(x,u,\nabla u) \; \de_{n+1} \zeta)\\
 & = - \int_{B_1^+} \langle F(x,u,\nabla u), \nabla \de_{n+1} \zeta \rangle
			- \int_{B_1^+} F_0(x,u,\nabla u) \; \de_{n+1} \zeta\\ 
&= \int_{B_1^+} \langle \de_{n+1} (F(x,u,\nabla u)), \nabla  \zeta \rangle
			- \int_{B_1^+} \de_{n+1} \langle F(x,u, \nabla u), \nabla \zeta \rangle\\
			&\qquad- \int_{B_1^+} F_0(x,u,\nabla u) \; \de_{n+1} \zeta\\
&= \int_{B_1^+} \langle a \nabla \de_{n+1} u - q, \nabla \zeta \rangle
			\, + \int_{B_1'} \langle F(x,u, \nabla u), \nabla \zeta \rangle,
\end{align*}
where we have set
\begin{gather*}
a(x) := D_pF(x,u(x),\nabla u(x)),\\
q(x) := - \de_{x_{n+1}} F(x,u(x),\nabla u(x)) - \de_z F(x,u(x),\nabla u(x)) \de_{n+1} u(x) + F_0(x,u(x),\nabla u(x))\; e_{n+1}.
\end{gather*}
We have hence inferred that 
\begin{gather}\label{e:quasi caccioppoli}
\int_{B_1^+} \langle a \nabla \de_{n+1} u - q, \nabla \zeta \rangle
\, + \int_{B_1'} \langle F'(x,u, \nabla u), \nabla' \zeta \rangle = 0.
\end{gather}
From the definition of $w=-F_{n+1}(x,0,0,\partial_{n+1}u(x))$, we get that
\begin{gather*}
\nabla \de_{n+1} u(x)
= - {1 \over \de_{p_{n+1}} F_{n+1}(x,0,0,\partial_{n+1}u(x))} \nabla w
- {\nabla_x F_{n+1}(x,0,0,\partial_{n+1}u(x)) \over \de_{p_{n+1}} F_{n+1}(x,0,0,\partial_{n+1}u(x))}.
\end{gather*}
Therefore we can rewrite equation \eqref{e:quasi caccioppoli} as
\begin{gather*}
\int_{B_1^+} \langle \tilde a \nabla w - \tilde q, \nabla \zeta \rangle
= \int_{B_1'} \langle F'(x,u, \nabla u), \nabla' \zeta \rangle,
\end{gather*}
where $\tilde a$ and $\tilde q$ are given by:
\begin{gather*}
\tilde a(x) := {1 \over \de_{p_{n+1}} F_{n+1}(x,0,0,\de_{n+1} u(x))} a(x),\\
\tilde q(x) := - q(x) - {1 \over \de_{p_{n+1}} F_{n+1}(x,0,0,\de_{n+1} u(x))} a(x) \nabla_x F_{n+1}(x,0,0,\de_{n+1} u(x)).
\end{gather*}
Since $\supp \zeta \;\cap\; B_1' \subset \subset \Lambda(u)$ and we assume \eqref{ip extra F gratis}, there exists a constant $C>0$ such that
\begin{align*}
\int_{B_1^+} \langle \tilde a \nabla w - \tilde q, \nabla \zeta \rangle
			& = \int_{B_1'} \langle F'(x',0,0,0,\de_{n+1} u(x',0)), \nabla' \zeta(x',0) \rangle \; dx'
			\\& = - \int_{B_1'} \div'(F'(x',0,0,0,\de_{n+1} u(x',0))) \; \zeta(x',0) \; dx'
			\\& = - \int_{B_1'} \div_{x'} F'(x',0,0,0,\de_{n+1} u(x',0)) \; \zeta(x',0) \; dx'
			\\& \le C \int_{B_1'} |\zeta|
			\le C \int_{B_1^+} |\nabla \zeta|,
\end{align*}
where in the last inequality we used the trace theorem for Sobolev functions $W^{1,1}(B_1^+)$.
Thus, we get the existence of a constant $c=c(\|u\|_\infty, \Lip(u))>0$ such that
\begin{equation}\label{eq con zeta}
\int_{B_1^+} \langle \tilde a \nabla w, \nabla \zeta \rangle \leq c\, \int_{B_1^+} |\nabla \zeta|,
\end{equation}
for all $\zeta\in C^\infty$ (and, hence, by a density argument for all $\zeta \in H^1(B_1^+)$)
with support satisfying the conditions \eqref{e:supporto zeta}.

Thus, we can consider $\zeta = \varphi^2 (w - k)_+$ with $k>0$ and $\varphi \in C_c^\infty(B_1^+)$ such that 
\[
\textup{
$\varphi \equiv 1$ on $B_r^+(x_0)$, $\varphi \equiv 0$ outside $B_R^+(x_0)$ and $|\nabla \varphi| \le {c \over R-r}$.}
\] 
Note that for $0 < k \le \|w\|_{L^\infty(B_1^+)}$ we have that $\{w > k\} \cap B_1'$ is open and compactly contained in the interior of $\Lambda(u)$ (in the relative topology of $B_1'$), because $u \in C^1(B_1^+ \cup B_1')$ by Theorem \ref{t.1}, 
therefore $\zeta$ satisfies the conditions \eqref{e:supporto zeta} on its support.
From standard computations we deduce \eqref{Caccioppoli n+1} for $k > 0$ (recall that the matrix $\tilde a$ is uniformly elliptic since $\lambda \le \de_{p_{n+1}} F_{n+1} \le L$).
Finally, we pass to the limit for $k \to 0^+$ to prove the inequality holds for $k = 0$ too, the case $k > \|w\|_{L^\infty(B_1^+)}$ being trivial.
\end{proof}

\subsection{H\"older continuity of the normal derivative}

Finally, we are ready to prove our second main result, Theorem \ref{t.c2}. The core of the proof is in Proposition \ref{p:core}, where we prove that the function
\[
\Phi u(x',0) = F_{n+1}((x',0), 0,0, \partial_{n+1} u(x',0))
\]
is H\"older continuous.
In what follows we denote by $\mathcal{H}^n$ the Hausdorff measure of dimension $n$.

\begin{proposition}\label{p:core}
	Let $u$ be a Lipschitz solution to the variational inequality \eqref{e.var-ineq-c2} for the boundary obstacle problem.
	Then, $\Phi u \in C^\beta(B_1^+ \cup B_1')$ for some $\beta\in (0,1)$.
\end{proposition}

\begin{proof}
Let $x_0 \in B_1'$ and $0 < r < 1 - |x_0|$. Then either
\begin{equation}\label{caso A}
\mathcal{H}^n(\Lambda(u) \cap B_{r/2}'(x_0)) \geq {1 \over 2} \mathcal{H}^n(B_{r/2}'(x_0))
\end{equation}
or
\begin{equation}\label{caso B}
\mathcal{H}^n(\{\Phi u = 0 \} \cap B_{r/2}'(x_0))
 \ge {1 \over 2} \mathcal{H}^n(B_{r/2}'(x_0)).
	\end{equation}
If \eqref{caso A} held, then
\begin{gather*}
\mathcal{H}^n(\{\de_i u = 0\} \cap B_{r/2}'(x_0)) \ge {1 \over 2} 
\mathcal{H}^n(B_{r/2}'(x_0))\qquad \forall\;i=1,\ldots, n.
\end{gather*}
Therefore, by De Giorgi's decay of the oscillation (see Theorem \ref{oscillazione} in the appendix) we have
\begin{gather}\label{stima osc1}
\osc_{B_{r/4}^+(x_0)} \de_i u \le \kappa \,\osc_{B_r^+(x_0)} \de_i u + c\,r
\qquad \forall\;i=1,\ldots, n.
\end{gather}
for some $\kappa = \kappa(n,M) \in (0,1)$.
Similarly, if \eqref{caso B} held, then 
\begin{gather}\label{stima osc2}
\osc_{B_{r/4}^+(x_0)} \Phi u \le \kappa \,\osc_{B_r^+(x_0)} \Phi u + c\,r.
\end{gather} 
We now follows Uralt'seva \cite{Uraltseva89}: we set $r_j := 4^{-j}r$, fix $\nu \ge 1$ and we consider the $2\nu+1$ radii $r_0,\dots,r_{2\nu}$. Then either \eqref{stima osc2} holds with $w = \Phi u$ for at least $\nu+1$ of these radii, or \eqref{stima osc1} holds with $w = \de_i u$, for every $i=1,\dots,n$ and for at least $\nu+1$ of these radii.
Let $r_{j_h}$, $h=0,\dots,\nu$ be the radii such that \eqref{stima osc2} holds with $w = \Phi u$. We label them so that $0 \le j_0 < j_1 < \dots < j_\nu \le 2\nu$, and notice then that $h \le j_h$ for every $h = 0,\dots,\nu$. We now set $\varphi(\rho) = \osc_{B_\rho^+(x_0)} w$ for every $0<\rho\le r$. We have
	\begin{gather*}
		\varphi(r_{j_{h+1}}) \le \varphi(r_{j_h +1}) \le \kappa \varphi(r_{j_h}) + 4^{-j_{h}}c\,r \le \kappa \varphi(r_{j_h}) + 4^{-h}c\,r
	\end{gather*}
for every $h = 0,\dots,\nu-1$. We then iterate the estimate to get
	\begin{gather*}
		\varphi(r_{2\nu}) \le \varphi(r_{j_\nu}) \le \kappa^\nu \left( \varphi(r) + {c\,r \over \kappa - {1 \over 4}} \right) \le \kappa^\nu (\varphi(r) + 4\,c\,r),
	\end{gather*}
where we have supposed without loss of generality that $\kappa \ge {1 \over 2}$. Note that the above inequality trivially holds for $\nu=0$. Hence, for $r_{2\nu+2} \le \rho < r_{2\nu}$ we have
	\begin{gather*}
		\osc_{B_\rho^+(x_0)} w = \varphi(\rho) \le \varphi(r_{2\nu}) \le \kappa^\nu (\varphi(r) + 4\,c\,r) \le \kappa^{-1} \left( {\rho \over r} \right)^{|\log_4 \kappa| \over 2} (\varphi(r) + 4\,c\,r) \le \\
		\le c\, (r^{-\beta} \osc_{B_r^+(x_0)} w + r^{1-\beta}) \rho^\beta \le c\, (r^{-\beta} \osc_{B_r^+(x_0)} w + 1) \rho^\beta,
	\end{gather*}
	where we have set $0<\beta = {|\log_4 \kappa| \over 2}<1$. So for every $\nu \ge 0$ we have that
	\begin{equation}\label{osc w}
		\osc_{B_\rho^+(x_0)} w \le c\, (r^{-\beta} \osc_{B_r^+(x_0)} w + 1) \rho^\beta,
	\end{equation}
	for every $r_{2\nu+2} \le \rho < r_{2\nu}$, either with $w = \Phi u$ or with $w = \de_i u$ for every $i=1,\dots,n$.
Now we set $k = \max\{0,\inf\limits_{B_\rho^+(x_0)} w\}$. By \eqref{Caccioppoli w} we get
	\begin{gather*}
		\int_{A\left(0,{\rho\over2}\right)} |\nabla w|^2 \le {c \over \rho^2} \int_{A(k,\rho)} (w-k)^2 + c\rho^{n+1} \le \\
		\le c\, \rho^{n+1} \left( \rho^{-2} \osc_{B_\rho^+(x_0)}^2 w + 1 \right).
	\end{gather*}
The same applies to $-w$, so summing up we have
\begin{gather*}
		\int_{B_{\rho / 2}^+(x_0)} |\nabla w|^2 \le c\, \rho^{n+1} \left( \rho^{-2} \osc_{B_\rho^+(x_0)}^2 w + 1 \right).
	\end{gather*}
Combined with \eqref{osc w}, this gives
\begin{gather}\label{e.morrey}
		\int_{B_{\rho / 2}^+(x_0)} |\nabla w|^2 \le c\, \rho^{n-1+2\beta} \left( r^{-2\beta} \osc_{B_r^+(x_0)}^2 w + 1 \right) =: C \rho^{n-1+2\beta},
	\end{gather}
for every $r_{2\nu+2} \le \rho < r_{2\nu}$, either with $w = \Phi u$ or with $w = \de_i u$ for every $i=1,\dots,n$. However $u$ satisfies an elliptic equation in $B_1^+$, so we can estimate $\de_{n+1}^2 u$ in terms of all the other second order derivatives of $u$. Thus we get
\begin{align*}
|\nabla \Phi u|^2
		&= |\nabla_x F_{n+1} + \de_{p_{n+1}} F_{n+1} \nabla \de_{n+1} u|^2\\
&\le c \; \left( |\nabla \de_{n+1} u|^2 + 1 \right)
\le c \left( \sum_{i = 1}^n |\nabla \de_i u|^2 + 1 \right)\quad \text{a.e. in }B_1^+.
	\end{align*}
Thus in any case
	\begin{gather*}
		\int_{B_{\rho / 2}^+(x_0)} |\nabla \Phi u|^2 \le C \rho^{n-1+2\beta}
		\qquad \forall\;\rho\in (0,1-|x_0|),
	\end{gather*}
independently of \eqref{e.morrey} holding for $\Phi u$ or for $\partial_iu$, $i=1,\ldots, n$.
By Morrey's theorem, we conclude that $\Phi u \in C^\beta(B_1^+)$ for some $\beta\in (0,1)$.
\end{proof}

\subsection{Proof of Theorem \ref{t.c2}}
Since $u$ solves the boundary obstacle problem, the co-normal derivative
\begin{gather*}
B_1'\ni x \mapsto F_{n+1}(x,u(x),\nabla u(x))
\end{gather*}
is continuous by Theorem \ref{t.c1}, vanishes on $B_1' \setminus \Lambda(u)$
and $Nu (x) = \Phi u (x)$ for every $x\in \Lambda(u)$ (because $u(x)=|\nabla' u(x)|=0$). Therefore,  $(Nu)(\cdot,0) \in C^\beta(B_1')$: indeed, 
for $(x',0) \in \Lambda(u)$ and $(x',0) \in B_1' \setminus \Lambda(u)$, there exists  $z' = (1-t) x' + t y \in \Gamma(u)$ for some $0 \leq t < 1$ such that
\begin{gather*}
|Nu(x') - Nu(y')| = |(Nu)(x',0)|= |\Phi u(x',0) - \Phi u (z',0)|
\le c\, |x'-z'|^\beta\le c\, |x'-y'|^\beta.
	\end{gather*}
We are now in the hypotheses to apply Theorem 2 of \cite{Lieberman88} (see also Theorem \ref{t.lieberman} in the appendix) and infer
$u \in C^{1,\alpha}(B^+_1 \cup B'_1)$ for some $\alpha \in (0,1)$.\qed

\appendix
\section{De Giorgi's oscillation lemma}
For readers' convenience, we report here De Giorgi's oscillation lemma \cite{DeGiorgi57} (e.g., we follow Chapter 7 of \cite{Giustibook} with small changes).
In this section $w$ is any function on $B_1^+$ satisfying a Caccioppoli inequality, namely
\begin{equation}\label{Caccioppoli Appendix}
	\int_{A(k,r)} |\nabla w|^2
	\le {Q \over (R-r)^2} \int_{A(k,R)} (w-k)^2
	+ Q\, |A(k,R)|
\end{equation}
for some $Q \geq 0$ and for every $k \ge 0$, $x_0 \in B_1'$, and $0<r<R<1 - |x_0|$. Throughout this section, every constant in the statements will depend on $n$ and on $Q$ unless specified.

The first consequence of inequality \eqref{Caccioppoli Appendix} is the following.

\begin{proposition}\label{prima conseguenza di Caccioppoli}
	There exists $c > 0$ such that
	\begin{equation}
		\int_{A(k,r)} (w-k)^2 \le
		c\, |A(k,R)|^{2 \over n + 1}
		\left(
		{1 \over (R-r)^2} \int_{A(k,R)} (w-k)^2
		+ c\, |A(k,R)|
		\right) 
	\end{equation}
for every $k \ge 0$, and $0 < r < R < 1 - |x_0|$.
\end{proposition}

\begin{proof}
	We fix $k \ge 0$ and $0<r<R<1 - |x_0|$ and set $r' := {r + R \over 2}$. Let $\varphi \in C^{\infty}(B_R^+(x_0))$ such that $\supp \varphi \subset\subset B_{r'}(x_0)$, $0 \le \varphi \le 1$, $\varphi \equiv 1$ on $B_r^+(x_0)$ and $|\nabla \varphi| \le {c \over R-r}$. If $n \ge 2$, we have
	\begin{gather*}
		\int_{A(k,r)} (w-k)^2 \le
		\int_{A(k,r')} (\varphi \, (w-k))^2 \le \\
		\le |A(k,R)|^{2 \over n + 1} \left( \int_{A(k,r')} (\varphi \, (w-k))^{2^*} \right)^{2 \over 2^*} \le
		c \, |A(k,R)|^{2 \over n + 1} \int_{A(k,r')} |\nabla(\varphi \, (w-k))|^2 \le \\
		\le c \, |A(k,R)|^{2 \over n + 1} \left( \int_{A(k,r')} (w-k)^2|\nabla \varphi|^2 + \int_{A(k,r')} \varphi^2|\nabla w|^2 \right) \le \\
		\le c \, |A(k,R)|^{2 \over n + 1} \left( {1 \over (R-r)^2} \int_{A(k,R)} (w-k)^2 + \int_{A(k,r')} |\nabla w|^2 \right) \le \\
		\le c \, |A(k,R)|^{2 \over n + 1} \left( {1 \over (R-r)^2} \int_{A(k,R)} (w-k)^2 \, +
		c\, |A(k,R)| \right).
\end{gather*}
If $n=1$ then $1^* = 2$, so in the above estimates we can replace the first two lines with
	\begin{gather*}
		\int_{A(k,r)} (w-k)^2 \le
		\int_{A(k,r')} (\varphi \, (w-k))^2 \le \\
		\le c\,\left( \int_{A(k,r')} |\nabla(\varphi \, (w-k))| \right)^2 \le
		c \, |A(k,R)| \int_{A(k,r')} |\nabla(\varphi \, (w-k))|^2
	\end{gather*}
and the rest of the proof is the same.
\end{proof}

\begin{proposition}\label{lim loc}
There exists $c > 0$ such that
\begin{gather*}
		\sup\limits_{B_{\rho / 2}^+(x_0)} w \le c \left( \fint_{A(k_0,\rho)} (w-k_0)^2 \right)^{1 \over 2} \left( {|A(k_0,\rho)| \over \rho^{n+1}} \right)^{\gamma \over 2} + k_0 + c \rho,
\end{gather*}
for every $k_0 \ge 0$ and $0 < \rho < 1 - |x_0|$, where $0<\gamma<1$ is such that $\gamma^2 + \gamma = {2 \over n+1}$.
\end{proposition}

\begin{proof} We set
\begin{gather*}
\phi(k,r) = |A(k,r)|^\gamma \int_{A(k,r)} (w-k)^2.
\end{gather*}
From Proposition \ref{prima conseguenza di Caccioppoli} we get
	\begin{gather*}
		\int_{A(k,r)} (w-k)^2 \le
		c \left( {1 \over (R-r)^2} + {1 \over (k-h)^2} \right) |A(k,R)|^{\gamma (1+ \gamma)} \int_{A(h,R)} (w-h)^2 \\
		|A(k,r)|^\gamma \le {1 \over (k-h)^{2\gamma}} \left( 	\int_{A(h,R)} (w-h)^2 \right)^\gamma \\
		\Longrightarrow \quad \phi(k,r) \le {c \over (k-h)^{2 \gamma}}  \left( {1 \over (R-r)^2} + {1 \over (k-h)^2} \right) \phi(h,R)^{1+\gamma}
\end{gather*}
for every $0 \le h<k$ and $0<r<R<1-|x_0|$. We then choose $k_j = h<k = k_{j+1}$ and $R_{j+1} = r<R = R_j$, $j \ge 0$, where
	\begin{gather*}
		k_j = k_0 + d \left( 1 - {1 \over 2^j} \right), \qquad
		R_j = {\rho \over 2} \left( 1 + {1 \over 2^j} \right).
	\end{gather*}
Here $k_0 \ge 0$, $0<\rho<1-|x_0|$ and $d > 0$ is a positive number to be chosen later. If $d \ge c \rho$, then 
	\begin{gather*}
		\phi(k_{j+1},R_{j+1}) \le {c (4^{1+\gamma})^j \over d^{2\gamma} \rho^2} \phi(k_j,R_j)^{1+\gamma} \left(1 + d^{-2}\rho^2 \right) \le DB^j \phi(k_j,R_j)^{1+\gamma},
\end{gather*}
where we have set $B = 4^{1+\gamma}$ and $D = c d^{-2\gamma} \rho^{-2}$. 
Moreover, if
\begin{gather*}
d \ge c \left( \fint_{A(k_0,\rho)} (w-k_0)^2 \right)^{1 \over 2} \left( {|A(k_0,\rho)| \over \rho^{n+1}} \right)^{\gamma \over 2},
\end{gather*}
then $\phi_0 \le D^{-{1 \over \gamma}} B^{-{1 \over \gamma^2}}$.

Now we apply the following fact (see \cite[Chapter 7]{Giustibook} for the simple proof by induction):
if $\lambda$, $D>0$, $B>1$, and $\phi_j$ is a sequence of positive real numbers such that
\begin{gather*}
\begin{cases}
\phi_{j+1} \le DB^j \phi_j^{1+\lambda} & \; \forall \; j \ge 0 \\
\phi_0 \le D^{-{1 \over \lambda}} B^{-{1 \over \lambda^2}}, &
\end{cases}
\qquad\Longrightarrow\qquad
\phi_j \le B^{-{j \over \lambda}} \phi_0 \quad\forall \;j \ge 0.
\end{gather*}
Therefore, if we consider
\begin{gather*}
d = c \left( \fint_{A(k_0,\rho)} (w-k_0)^2 \right)^{1 \over 2} \left( {|A(k_0,\rho)| \over \rho^{n+1}} \right)^{\gamma \over 2} + c \rho,
\qquad \phi_j = \phi(k_j,R_j),
\end{gather*}
we get
\begin{gather*}
4^j |A(k_j,R_j)|^\gamma \int_{A(k_j,R_j)} (w-k_j)^2 \le 4^{-{j \over \gamma}} \phi_0
\end{gather*}
which yields
\begin{gather*}
|A(d + k_0,\rho /2)|^{1+\gamma} \le {4^j \over d^2} |A(k_j,R_j)|^\gamma \int_{A(k_j,R_j)} (w-k_j)^2 \le 4^{-{j \over \gamma}} {\phi_0 \over d^2} \qquad \forall j \ge 0,
\end{gather*}
so $|A(d + k_0,\rho /2)|=0$, i.e. $w \le d + k_0$ a.e. in $B_{\rho /2}^+(x_0)$. 
\end{proof}

In the following, we set
\begin{gather*}
	M(r) = \sup\limits_{B_r^+(x_0)} w, \quad
	m(r) = \inf\limits_{B_r^+(x_0)} w, \\
	\osc(r) = M(r) - m(r)
\end{gather*}
for every $0<r<1-|x_0|$.

\begin{proposition}\label{insiemi di livello}
There exists $C > 0$ such that, if
\begin{gather*}
\mathcal{H}^n(\{w = 0\} \cap B_{r/2}'(x_0))_n \geq {1 \over 2} 
\mathcal{H}^n(B_{r/2}'(x_0)), \qquad
{M(r) + m(r) \over 2} \ge 0, \\
\osc(r) \ge 2^{N-1} r
\end{gather*}
for some $0<r<1-|x_0|$ and $N \ge 1$, then if we set
\begin{gather*}
k_j = M(r) - 2^{-j-1}\osc(r)
\end{gather*}
for every $j \ge 0$, we have
\begin{gather*}
{|A(k_N,r/2)| \over r^{n+1}} \le C N^{-{n+1 \over 2n}}.
\end{gather*}
\end{proposition}

\begin{proof}
For every $0 \le h < k$ with $h \le M(r) - {1 \over 2} r$, we let $\overline w$ be the function defined in $B_{r/2}^+(x_0)$ by the law
	\begin{gather*}
		\overline w(x)=\begin{cases}
			k-h & \mbox{if } w(x) \ge k \\
			w(x)-h & \mbox{if } h \le w(x) \le k \\
			0 & \mbox{if } w(x) \le h.
		\end{cases}
	\end{gather*}
	Since
	\begin{gather*}
		\mathcal{H}^n(\{\overline w = 0\} \cap B_{r/2}'(x_0)) \ge \left|\{w = 0\} \cap B_{r/2}'(x_0)\right|_n \ge {1 \over 2} \mathcal{H}^n(B_{r/2}'(x_0))
	\end{gather*}
	we may use the Sobolev-Poincaré inequality to deduce
	\begin{gather*}
		(k-h)|A(k,r/2)|^{1 \over 1^*} \le \left( \int_{B_{r/2}^+(x_0)} \overline w^{1^*} \right)^{1 \over 1^*} \le c \int_{B_{r/2}^+(x_0)} |\nabla \overline w| = \\
		= c \int_\Delta |\nabla w| \le c \, |\Delta|^{1 \over 2} \left( \int_\Delta |\nabla w|^2 \right)^{1 \over 2}
	\end{gather*}
where we have set $\Delta = A(h,r/2) \setminus A(k,r/2)$. From \eqref{Caccioppoli w} we have
	\begin{gather*}
		\int_\Delta |\nabla w|^2 \le c \left( {1 \over r^2} \int_{A(h,r)} (w-h)^2 + |A(h,r)| \right) \le \\
		\le c\,r^{n-1} \left( (M(r)-h)^2 + r^2 \right) \le c\,r^{n-1} (M(r)-h)^2.
	\end{gather*}
Thus
	\begin{gather*}
		(k-h)^2 |A(k,r/2)|^{2 \over 1^*} \le c\,r^{n-1} |\Delta| (M(r)-h)^2.
	\end{gather*}
Now we choose $h = k_{j-1} < k = k_j$ for every $j = 1,\dots,N$. We set $A_j := |A(k_j,r/2)|$. Since $\osc(r)>0$, we get
\begin{gather*}
A_N^{2 \over 1^*} \le A_j^{2 \over 1^*} \le c\,r^{n-1} (A_{j-1} - A_j)\qquad
\forall\; i=1, \ldots, N.
\end{gather*}
Finally we sum on $j$ to get
\begin{gather*}
N A_N^{2 \over 1^*} \le c\,r^{n-1} (A_0 - A_N) \le c\,r^{n-1} A_0 \le c\,r^{2n}
\end{gather*}
which yields the conclusion.
\end{proof}

Finally, the following is the De Giorgi's oscillation lemma.

\begin{theorem}\label{oscillazione}
There exist $0 < \kappa < 1$ and $c>0$ such that, if
\begin{gather*}
\mathcal{H}^n(\{w = 0\} \cap B_{r/2}'(x_0)) \ge {1 \over 2} \mathcal{H}^n(B_{r/2}'(x_0)),
\end{gather*}
then
\begin{gather*}
\osc_{B_{r/4}^+(x_0)} w \le \kappa \,\osc_{B_r^+(x_0)} w + cr
\qquad \forall\; 0 < r < 1 - |x_0|.
\end{gather*}
\end{theorem}

\begin{proof}
With the notations of the preceding proof, without loss of generality we may assume that ${M(r) + m(r) \over 2} \ge 0$, since otherwise we can replace $w$ with $-w$. By Proposition \ref{lim loc} and \ref{insiemi di livello} with $N \ge 1$ to be chosen later, we get
\begin{gather*}
M(r/4) - k_N \le c \left( \fint_{A(k_N,r/2)} (w-k_N)^2 \right)^{1 \over 2} \left( {|A(k_N,r/2)| \over r^{n+1}} \right)^{\gamma \over 2} + 
cr
\le \\
\le c\, (M(r)-k_N) \left( {|A(k_N,r/2)| \over r^{n+1}} \right)^{1+\gamma \over 2} + {1 \over 2} r \le c\, (M(r)-k_N) N^{-{1 \over 2n\gamma}} + cr
\end{gather*}
if $\osc(r) \ge 2^{N-1} r$. So if we choose $N \ge 1$ big enough to have $c N^{-{1 \over 2n\gamma}} \le {1 \over 2}$, we get
\begin{gather*}
M(r/4) - k_N \le {1 \over 2} (M(r) - k_N) + cr.
\end{gather*}
By the very definition of $k_N$ and of oscillation, with elementary passages we come to
\begin{gather*}
\osc(r/4) \le (1-2^{-N-2}) \,\osc(r) + cr.
\end{gather*}
If instead we had had $\osc(r) \le 2^{N-1} r$, then we would have had
\begin{gather*}
\osc(r/4) \le \osc(r) = (1-2^{-N-2}) \,\osc(r) + \,2^{-N-2} \osc(r) \le (1-2^{-N-2}) \,\osc(r) + {1 \over 8} r
\end{gather*}
so we finish the proof by setting $\kappa = 1-2^{-N-2}$.
\end{proof}

\section{Regularity and Harnack's inequality}\label{a.regolarità}
In this appendix we recall the regularity theorems we have used throughout the paper. We still use notations and hypotheses on the functions $F$ and $F_0$ as in Sections \ref{Introduction} and \ref{s.rettificare}.


\begin{theorem}[{\cite[Theorem 1.2]{Trudinger67}}]\label{Appendix-Trudinger}
Let $u$ be a weak supersolution of
\begin{gather*}
\div A(x, u(x), \nabla u(x)) = B(x, u(x), \nabla u(x))\qquad \forall \; x \in B_{3r}(x_0),
\end{gather*}
such that $0 \le u < M$ in $B_{3r}(x_0)$, $x_0 \in \R^{n+1}$, $r,M > 0$
and 
\begin{gather*}
|A(x, z, p)| + |B(x, z, p)| \leq C_0 |p| + C_0 |z|,\qquad \langle A(x, z, p), p\rangle \geq |p|^2-C_0 u^2.
\end{gather*}
Then
\begin{equation*}
\left( \fint_{B_{2r}(x_0)} u^q \right)^{1 \over q} \le C \min\limits_{B_r(x_0)} u(x)
\end{equation*}
for any $1 \le q < {n \over n - 2}$ if $2 \le n$, and for any $1 \le q \le \infty$ if $n<2$, and for some $C = C(n, C_0, q, M )$.
\end{theorem}

From this result we deduce the following corollary.

\begin{corollary}\label{c.harnack}
Let $v$, $w \in \Lip(B_{3r}(x_0))$ be respectively a weak subsolution and a weak supersolution of
\begin{gather}\label{eq-harnack}
\div F(x, u(x), \nabla u(x)) = F_0(x, u(x), \nabla u(x))
\qquad \forall \; x \in B_{3r}(x_0)
\end{gather}
Suppose that $v \le w$ in $B_{3r}(x_0)$ and $v(x_0) = w(x_0)$. Then $v \equiv w$ in $B_r(x_0)$.
\end{corollary}

\begin{proof}
We have that, for every $\varphi \in C^\infty_0(B_{3r}(x_0)), \,\varphi \ge 0$,
\begin{gather*}
\int_{B_{3r}(x_0)} \langle F(x,w,\nabla w)-F(x,v,\nabla v),\nabla\varphi \rangle + (F_0(x,w,\nabla w)-F_0(x,v,\nabla v)) \varphi \ge 0.
\end{gather*}
Write $u_t = t w + (1-t) v$ for $0\le t \le 1$ and $u = w-v$. Then we have
\begin{gather*}
F(x,w,\nabla w)-F(x,v,\nabla v) = a(x)\nabla u + d(x) u, \\
F_0(x,w,\nabla w)-F_0(x,v,\nabla v) = \langle b(x), \nabla u \rangle + c(x) u,
\end{gather*}
where
\begin{gather*}
a_{ij}(x) = \int_0^1  \de_{p_j} F_i(x,u_t,\nabla u_t) \,dt, 
\qquad d_i(x) = \int_0^1 \de_z F_i(x,u_t,\nabla u_t)\,dt,\\
b_j(x)=\int_0^1 \de_{p_j} F_0(x,u_t,\nabla u_t) \,dt, \qquad c(x) = \int_0^1 \de_z F_0(x,u_t,\nabla u_t)\,dt
\end{gather*}
are continuous functions.
Thus, we have that $u \ge 0$ is a weak supersolution of 
\begin{gather*}
Hu(x) = - \div(a(x) \nabla u(x) + d(x)u(x)) + \langle b(x), \nabla u(x) \rangle + c(x) u(x).
\end{gather*}
By Theorem \ref{Appendix-Trudinger}, applied to $A(x,z,p)= a(x)p + d(x) z$ and $B(x,z,p)=\langle b(x), p\rangle + c(x)z$, since $u(x_0) = 0$ we have that $u \equiv 0$ in $B_r(x_0)$, which concludes the proof.
\end{proof}

We also recall the boundary regularity for both the Dirichlet and the Neumann problem.


\begin{theorem}[{\cite[Theorem A]{GiaquintaGiusti84}}]
Let $u$ be a bounded Lipschitz weak solution of the Dirichlet problem
\begin{gather*}
\begin{cases}
\div F(x, u(x), \nabla u(x)) = F_0(x, u(x), \nabla u(x))
\qquad & \forall \; x \in B^+_r(x_0) \\
u(x) = 0 & \forall \; x \in B'_r(x_0)
\end{cases}
\end{gather*}
such that $|u| \le M$ in $B^+_r(x_0)$, where $x_0 \in \R^n \times \{0\}$, $r > 0$. Then $u \in C^{1, \alpha}(B^+_r(x_0) \cup B'_r(x_0))$ for some $0 < \alpha = \alpha(n, M, \lambda, L) < 1$ and norm $\|u\|_{1+\alpha}\leq C=\alpha(n, M, \Lip(u), \lambda, L)$.
\end{theorem}


\begin{theorem}[{\cite[Theorem 2]{Lieberman88}}]\label{t.lieberman}
Let $u$ be a bounded Lipschitz weak solution of Neumann problem
\begin{gather*}
\begin{cases}
\div F(x, u(x), \nabla u(x)) = F_0(x, u(x), \nabla u(x))
\qquad & \forall \; x \in B^+_r(x_0), \\
F_{n+1}(x, u(x), \nabla u(x)) = 0 & \forall \; x \in B'_r(x_0),
\end{cases}
\end{gather*}
such that $|u| \le M$ in $B^+_r(x_0)$, where $x_0 \in \R^n \times \{0\}$, $r > 0$. Then $u \in C^{1, \alpha}(B^+_r(x_0) \cup B'_r(x_0))$ for some $0 < \alpha = \alpha(n, M, \lambda, L) < 1$ and norm $\|u\|_{1+\alpha}\leq C=\alpha(n, M, \Lip(u), \lambda, L)$.
\end{theorem}

%
%

\bibliographystyle{plain}

\end{document}